\documentclass[11pt, a4paper]{article}

\usepackage{jheppub}
\author[a]{Mateo Galdeano,}
\author[b]{Leander Stecker}
\affiliation[a]{Fachbereich Mathematik, Universit\"at Hamburg\\ Bundesstra{\ss}e 55, 20146, Hamburg, Germany}
\affiliation[b]{Fachbereich Mathematik und Informatik, Philipps-Universität Marburg\\ Campus Lahnberge, 35032 Marburg, Germany}
\emailAdd{mateo.galdeano@uni-hamburg.de}
\emailAdd{stecker@mathematik.uni-marburg.de}


\usepackage{amsthm}
\usepackage{enumerate}

\usepackage{mathtools}

\usepackage[nameinlink]{cleveref}

\usepackage{physics}
\usepackage{tensor}

\usepackage{multirow}
\usepackage{diagbox}

\usepackage{tikz} 
\usetikzlibrary{arrows} 

\newcommand*\N{\mathbb{N}}

\newcommand*\R{\mathbb{R}}
\newcommand*\C{\mathbb{C}}

\newcommand{\cyclic}[1]{\stackrel{\scriptsize #1}{\mathfrak{S}}} 

\newtheorem{thm}{Theorem}[section]
\newtheorem{cor}[thm]{Corollary}
\newtheorem{prop}[thm]{Proposition}
\newtheorem{lem}[thm]{Lemma}
\theoremstyle{definition} 
\newtheorem{defi}[thm]{Definition}

\newtheorem{rem}[thm]{Remark}
\newtheorem{rems}[thm]{Remarks}
\newtheorem{ex}[thm]{Example}

\newif\ifcomments
\commentstrue

\newif\ifdetails
\detailstrue

\begin{document}

\title{The heterotic G$_2$ system with reducible characteristic holonomy}

\abstract{
We construct solutions to the heterotic G$_2$ system on almost contact metric manifolds with reduced characteristic holonomy. We focus on $3$-$(\alpha,\delta)$-Sasaki manifolds and $(\alpha,\delta)$-Sasaki manifolds, the latter being a convenient reformulation of spin $\eta$-Einstein $\alpha$-Sasaki manifolds. Investigating a $1$-parameter family of G$_2$-connections on the tangent bundle, we obtain several approximate solutions as well as one new class of exact solutions on degenerate $3$-$(\alpha,\delta)$-Sasaki manifolds.
}

\maketitle

\section{Introduction}

In this paper we study the existence of solutions to the heterotic G$_2$ system on certain 7-dimensional almost contact metric manifolds for which the characteristic holonomy is reduced to either $\mathrm{SU}(3)$ or $\mathrm{Sp}(1)\mathrm{Sp}(1)$.

The \emph{heterotic G$_2$ system} \cite{Gunaydin:1995ku,Gauntlett:2001ur, Friedrich:2001nh, Friedrich:2001yp, Gauntlett:2002sc, Gauntlett:2003cy, Ivanov:2003nd, Ivanov:2009rh, Kunitomo:2009mx, Lukas:2010mf, Gray:2012md} describes supersymmetric compactifications of heterotic supergravity---including first order $\alpha'$ corrections---on a warped product of a compact 7-dimensional manifold and either 3-dimensional Minkowski or anti-de Sitter space. It can be regarded as the 7-dimensional analogue of the Hull--Strominger system \cite{Strominger:1986uh, Hull:1986kz} and it is sometimes referred to in the literature as the \emph{G$_2$ Hull--Strominger system}. Solutions of the heterotic G$_2$ system have been constructed in \cite{Gunaydin:1995ku, Fernandez:2008wla, Nolle:2010nn, Fernandez:2014pfa, Clarke:2020erl, Lotay:2021eog, delaOssa:2021qlt}.

Any manifold $M$ with a G$_2$-structure has an underlying almost contact metric structure as well as an almost $3$-contact metric structure \cite{delaOssa:2021cgd}, which amount to a reduction of the structure group of $M$ to $\mathrm{SU}(3)$ or $\mathrm{SU}(2)$, respectively. We use these additional structures to construct new solutions to the heterotic G$_2$ system. In particular, we will repeatedly use a decomposition of the tangent bundle $TM$ determined by the action of the structure groups. 

These $G$-structures are particularly well-behaved when they admit a characteristic connection with parallel skew torsion and reduced holonomy. Thus, we focus on manifolds where the characteristic holonomy reduces to one of the subgroups of G$_2$ depicted in \Cref{fig:Diagram}. 
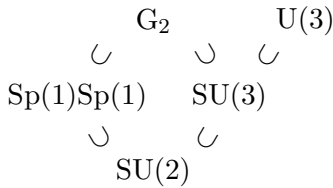
\begin{figure}[h]
\centering
\begin{tikzpicture}
\node at (-0.7,0.55) {\rotatebox{60}{$\subset$}};
\node at (0.7,0.55) {\rotatebox{120}{$\subset$}};
\node at (-0.7,-0.55) {\rotatebox{120}{$\subset$}};
\node at (0.7,-0.55) {\rotatebox{60}{$\subset$}};
\node at (0,1) {$\mathrm{G}_2$};
\node at (-1,0) {$\mathrm{Sp}(1)\mathrm{Sp}(1)$};
\node at (1,0) {$\mathrm{SU}(3)$};
\node at (0,-1) {$\mathrm{SU}(2)$};
\node at (1.5,0.55) {\rotatebox{60}{$\subset$}};
\node at (2,1) {$\mathrm{U}(3)$};
\end{tikzpicture}
\caption{Diagram of groups.}
\label{fig:Diagram}
\end{figure}

In our solutions, the compact manifold $M$ can always be described as the total space of a Riemannian foliation. This is a common trait of all the existing solutions of the heterotic G$_2$ system. In fact, although our solutions differ from those in \cite{Clarke:2020erl, Lotay:2021eog, delaOssa:2021qlt} by our choice of $\mathrm{G}_2$ instantons, the manifolds underlying those constructions can be regarded as particular cases of the manifolds we consider.

Under our assumptions, the Riemannian submersion from $M$ to the leaf space is compatible with the characteristic G$_2$-connection. Solving the heterotic G$_2$ system requires a detailed understanding of this connection as well as two different instanton connections on bundles over $M$. We show that the study of the heterotic G$_2$ system is much simpler when all the connections can be projected to the base of the submersion. For that reason we will consider instanton connections on the tangent bundle that lie in a $1$-parameter family $\nabla^\lambda$ and satisfy this property.

We first study the case of characteristic holonomy $\mathrm{Sp}(1)\mathrm{Sp}(1)$. A particularly nice class---and to the authors knowledge the only one known---of manifolds satisfying all these assumptions is the class of 3-$(\alpha,\delta)$-Sasaki manifolds. These were defined in \cite{Agricola:2018} as generalizations of $3$-Sasaki structures admitting well-behaved characteristic connections. We find approximate solutions as well as the following exact solutions in the degenerate case $\delta=0$:
\begin{thm}\label{thm:exactsolution}
Let $\alpha'>0$ and let $(M,g,\xi_i,\eta_i,\phi_i)_{i=1,2,3}$ be a degenerate $7$-dimensional $3$-$(\alpha,\delta)$-Sasaki manifold with its canonical G$_2$-structure $\varphi$ and canonical connection $\nabla$ with torsion $T$. If $\alpha^2=\frac{1}{12\alpha'}$, then 
\begin{equation*}
[(M,\varphi),(TM,\nabla^{-\beta}),(TM,\nabla), T]\, ,
\end{equation*}
where $\beta=2(\delta-2\alpha)$, is a solution to the heterotic G$_2$ system.
\end{thm}

For the $\mathrm{SU}(3)$ case, the best analogue is given by $\eta$-Einstein $\alpha$-Sasaki manifolds. When these manifolds admit a spin structure, they also depend on two parameters $(\alpha,\delta)$ and we will call them $(\alpha,\delta)$-Sasaki manifolds. This is motivated by the fact that the dependence on the parameters $(\alpha,\delta)$ is very similar to that of 3-$(\alpha,\delta)$-Sasaki manifolds. Unfortunately, this case is less well-behaved and we will only be able to obtain approximate solutions.

Both 3-$(\alpha,\delta)$-Sasaki and $(\alpha,\delta)$-Sasaki manifolds share an additional common feature: they are characterized by the existence of certain generalized Killing spinors. We will also present this formalism and use it to give a description of the different G$_2$-structures from a spinorial point of view. 

The paper is organized as follows. In \Cref{sec:background} we introduce the heterotic G$_2$ system and in \Cref{sec:Contact} the geometries where we are going to study the system. In \Cref{sec:su2case} we discuss solutions on 3-$(\alpha,\delta)$-Sasaki manifolds. In \Cref{sec:su3case} we discuss the case of spin $\eta$-Einstein $\alpha$-Sasaki manifolds and introduce our description as $(\alpha,\delta)$-Sasaki manifolds. Finally, in \Cref{sec:conclusion} we summarize our results and point out some future directions. \Cref{app:spinors} gives an overview of the spinorial perspective.

\section{The heterotic \texorpdfstring{G$_2$}{G2} system}
\label{sec:background}

In this section we introduce the heterotic G${}_2$ system following \cite{delaOssa:2017pqy}. The system describes $\mathcal{N}=1$ supersymmetric vacuum solutions of heterotic supergravity on a warped product $\mathcal{M}_{3}\times M$, where $\mathcal{M}_{3}$ is a maximally symmetric 3-dimensional Lorentzian space (the \emph{spacetime}) and $M$ is a compact 7-dimensional Riemannian manifold.

Supersymmetry requires the compact manifold $M$ to have a G$_2$-structure and the gauge fields to be G$_2$-instantons, so we begin with a brief review of G$_2$-geometry to present the concepts that we will use and set our notation. For more detailed introductions to the topic, see for example \cite{Joyce:2000, Bryant:2005mz, Karigiannis:2020}.

\begin{defi}
    Let $M$ be a 7-dimensional manifold. A \emph{$\mathrm{G}_2$-structure} on $M$ is a nowhere-vanishing three-form $\varphi$ on $M$ which can be identified at every point $p\in M$ with
\begin{equation*}
\label{eq:varphi0}
\varphi_0=e^{123}-e^{145}-e^{167}-e^{246}+e^{257}-e^{347}-e^{356} \, ,
\end{equation*}
where $\lbrace e^1,\dots,e^7\rbrace$ is a local basis of one-forms and we are writing $e^{ij}=e^i\wedge e^j \, $. We call $\varphi$ the \emph{associative form} and we say that $(M,\varphi)$ is a \emph{manifold with a $\mathrm{G}_2$-structure}.
\end{defi}
\begin{rem}
    The existence of the associative three-form $\varphi$ is equivalent to the usual notion of G$_2$-structure as a reduction of the structure group to G$_2\,$. Furthermore, $\varphi$ determines a metric $g$ and an orientation on $M$ which can be used to define the \emph{coassociative} four-form
    \begin{equation*}
        \psi=*\varphi \, .
    \end{equation*}
\end{rem}
Tensors in $(M,\varphi)$ decompose in terms of representations of the group G${}_2\,$. The following proposition shows the decomposition in the case of differential forms.
\begin{prop}
    Consider $(M,\varphi)$ a manifold with a $\mathrm{G}_2$-structure. Let $\Lambda^k=\Lambda^k(T^*M)$ denote the bundle of $k$-forms on $M$ and $\Lambda^k_p=\Lambda^k_p(T^*M)$ denote the subbundle of $\Lambda^k$ consisting of $k$-forms transforming in the $p$-dimensional irreducible representation of G$_2\,$. Then, the following holds
    \begin{equation*}
        \label{eq:splittingofformseq}
        \Lambda^0=\Lambda^0_1 \, ,\qquad \Lambda^1=\Lambda^1_7 \, ,\qquad \Lambda^2=\Lambda^2_7\oplus\Lambda^2_{14} \, ,\qquad \Lambda^3=\Lambda^3_1\oplus\Lambda^3_7\oplus\Lambda^3_{27} \, ,
    \end{equation*}
    and the decomposition for higher degrees follows from Hodge duality.
\end{prop}
Each $\Lambda^k_p$ can be characterized in terms of the associative and coassociative forms. For example $\Lambda^2_{14} \, $, which corresponds to the two-forms contained in the Lie algebra $\mathfrak{g}_2$ of G${}_2\,$, can be described as
\begin{equation}
\label{eq:splitofg2liealgebra}
\Lambda^2_{14}=\lbrace\beta\in\Lambda^2 : \beta\lrcorner \, \varphi=0\rbrace=\lbrace\beta\in\Lambda^2 : \beta\wedge\psi=0\rbrace \, .
\end{equation}
The decomposition of $\dd\varphi$ and $\dd\psi$ in G${}_2$-representations characterizes the G$_2$-structure.
\begin{defi}
    Let $(M,\varphi)$ be a manifold with a $\mathrm{G}_2$-structure. The \emph{torsion classes} of the $\mathrm{G}_2$-structure are the forms $\tau_0\in\Lambda^0$, $\tau_1\in\Lambda^1$, $\tau_2\in\Lambda^2_{14}$ and $\tau_3\in\Lambda^3_{27}$ satisfying
    \begin{equation*}
        \label{eq:torsionclassequation}
        \dd\varphi = \tau_0 \, \psi+3\,\tau_1\wedge\varphi+*\tau_3 \, , \qquad
        \dd\psi = 4\,\tau_1\wedge\psi+*\tau_2 \, .
    \end{equation*}
\end{defi}
When all the torsion classes vanish the G$_2$-structure is \emph{torsion-free} and the manifold $M$ has G$_2$ \emph{holonomy}. We are interested in more general types of G$_2$-structures.
\begin{defi}
     Let $(M,\varphi)$ be a manifold with a $\mathrm{G}_2$-structure. We say the $\mathrm{G}_2$-structure is \emph{conformally coclosed} if $\tau_2=0$, and we say that it is \emph{coclosed} if $\tau_1=\tau_2=0$.
\end{defi}
Conformally coclosed $\mathrm{G}_2$-structures have the key property of admitting a characteristic G$_2$-connection. We now recall the necessary notions: metric connections $\nabla$ on $TM$ are in bijective correspondence to torsion tensors
\begin{equation*}
    T(X,Y,Z)\coloneqq g(X,\nabla_YZ-\nabla_ZY-[Y,Z]) \, .
\end{equation*}
Indeed, for a torsion tensor $T$ the corresponding connection is given by $g(X,\nabla_YZ)=g(X,\nabla^g_YZ)+A(X,Y,Z)$, where $A$ is the contorsion tensor defined as
\begin{equation*}\label{eq:contorsion}
    A(X,Y,Z)\coloneqq\frac 12(T(X,Y,Z)-T(Y,Z,X)+T(Z,X,Y)) \, .
\end{equation*}
\begin{defi}
    A metric connection is said to have skew-symmetric torsion (or simply skew torsion) if $T\in\Lambda^3T^*M$.
\end{defi}
In this case $A=\frac12 T\,$, so in short $\nabla=\nabla^g+\frac 12 T$. In \cite{Friedrich:2001nh} the authors prove
\begin{prop}[\cite{Friedrich:2001nh}]
    Let $(M,\varphi)$ be a manifold with a conformally coclosed $\mathrm{G}_2$-structure. There exists a unique metric connection $\nabla^c$ that has totally skew-symmetric torsion and is compatible with the G$_2$-structure. Its torsion tensor is given by
\begin{equation}
\label{eq:torsiong2}
T^c=\frac{1}{6}\,\tau_0 \, \varphi-\tau_1\lrcorner \, \psi-\tau_3 \, .
\end{equation}
We call this the \emph{characteristic} G$_2$-connection.
\end{prop}
In addition, we will be interested in connections that are related to the G$_2$-structure in a particular way:
\begin{defi}
    Let $A$ be a connection on a vector bundle $V$ with curvature $F_A\in \Lambda^2(\operatorname{End}(V))$. We say $A$ is a G$_2$-\emph{instanton} if its curvature satisfies $F_A\in \Lambda^2_{14}(\operatorname{End}(V))$. Using \eqref{eq:splitofg2liealgebra}, this can be equivalently expressed as
\begin{equation*}
F_A\wedge\psi=0\, .
\end{equation*}
\end{defi}
We are finally in a position to introduce our main object of study:
\begin{defi}
\label{def:heteroticG2system}
    A quadruple $[(M,\varphi),(V,A),(TM,\Theta),H]$ constitutes a  \emph{solution to the heterotic G${}_2$ system} if the following holds:
    \begin{itemize}
        \item $(M,\varphi)$ is a 7-dimensional compact manifold with a conformally coclosed G$_2$-structure.
        \item $H=T^c$ is the torsion of the characteristic G$_2$-connection.
        \item $V$ is a vector bundle on $M$ equipped with a G$_2$-instanton connection $A$.
        \item The tangent bundle $TM$ of $M$ is equipped with a G$_2$-instanton connection $\Theta$.
    \end{itemize}
    In addition, the physical flux $H$ and the instanton connections $A$, $\Theta$ must satisfy the \emph{heterotic Bianchi identity}
    \begin{equation}
        \label{eq:heteroticBianchiidentity}
        \dd H=\frac{\alpha'}{4}\,(\tr F_A\wedge F_A -\tr \mathcal{R}_\Theta\wedge \mathcal{R}_\Theta) \, ,
    \end{equation}
    where $\alpha'>0$ is the \emph{string parameter}, $F_A$ denotes the curvature of $A$ and $\mathcal{R}_\Theta$ is the curvature of $\Theta$.
\end{defi}
\begin{rem}
    A quadruple $[(M,\varphi),(V,A),(TM,\Theta),H]$ satisfying the first three bullet points in \Cref{def:heteroticG2system} is, from the point of view of heterotic supergravity, a solution to the \emph{Killing spinor equations} that any supersymmetric background requires. The heterotic Bianchi identity is imposed by the anomaly cancellation condition of heterotic string theory, and---provided that the fourth bullet point holds---this is sufficient to ensure that a solution to the Killing spinor equations is in fact a solution to the equations of motion of heterotic supergravity \cite{Ivanov:2009rh}. Thus, a solution to the heterotic G$_2$ system constitutes a supersymmetric solution of heterotic supergravity.
\end{rem}

    Heterotic supergravity in this context should be understood as the theory obtained at first order via a perturbative expansion of heterotic string theory on the string parameter $\alpha'$ \cite{Bergshoeff:1988nn, Bergshoeff:1989de}. As pointed out in \cite{delaOssa:2014msa}, this means that it is enough for the equations to be satisfied only \emph{up to first order in $\alpha'$}.

    We can therefore consider a relaxed version of the heterotic G$_2$ system. We still require the first three bullet points in \Cref{def:heteroticG2system} and the Bianchi identity to hold exactly, but we impose the instanton condition on $\Theta$ only up to first order in $\alpha'$. Given a quadruple $[(M, \varphi(\alpha')), (V,A(\alpha')), (TM,\Theta(\alpha')), H(\alpha')]$ where both gauge fields and the G$_2$-structure may depend on $\alpha'$, this condition can be phrased following \cite{Lotay:2021eog} as
    \begin{equation}
    \label{eq:approximateG2instantoncondition}
        \vert \mathcal{R}_\Theta\wedge\psi\vert_g=\mathcal{O}(\alpha')^2  \qquad \text{as } \alpha'\to 0 \, ,
    \end{equation}
    where $\vert\cdot\vert_g$ is the pointwise $C^0$-norm with respect to the metric $g$ induced by $\varphi$ and $\mathcal{R}_\Theta$ is regarded as a section of $\Lambda^2\otimes\mathrm{End}(TM)$.    
\begin{defi}
\label{def:approximate}
    We say that a solution to the heterotic G$_2$ system is \emph{exact} if $\Theta$ is an honest G$_2$-instanton, and \emph{approximate} if $\Theta$ instead satisfies \eqref{eq:approximateG2instantoncondition}.
\end{defi}
\begin{rem}
\label{rem:cosmologicalconstant}
    The torsion classes of the G$_2$-structure present in the heterotic G$_2$ system encode some of the physical information of the compactification \cite{delaOssa:2019cci}. In particular, the dilaton field $\mu$ is encoded by $\tau_1=\frac{1}{2}\,\dd\mu $ and the cosmological constant $\Lambda$ of the 3-dimensional spacetime satisfies $\Lambda\sim -\,\tau_0^2$.
\end{rem}
In order to find solutions below we are particularly interested in connections with skew-torsion that satisfy $\nabla T=0$, where $T$ is the torsion of $\nabla$. In this case we say $\nabla$ has \emph{parallel skew-torsion}. This assumption has important implications on the curvature
\begin{equation*}
    R(X,Y,Z,V)\coloneqq g(([\nabla_X,\nabla_Y]-\nabla_{[X,Y]})Z,V)\, .
\end{equation*}
Indeed, in this case we obtain a tangible first Bianchi identity
\begin{equation}\label{torsionBianchi}
    \cyclic{X,Y,Z}R(X,Y,Z,V)=\sigma_T(X,Y,Z,V)\coloneqq \cyclic{X,Y,Z}g(T(X,Y),T(Z,V))\, ,
\end{equation}
where here and hereafter $\cyclic{X,Y,Z}$ denotes the sum over all cyclic permutations of $X,Y,Z$. In particular, this yields pair symmetry
\begin{equation*}
    R(X,Y,Z,V)=R(Z,V,X,Y)\, .
\end{equation*}
\begin{rem}
    Note that if the characteristic connection $\nabla^c$ has parallel skew-torsion then pair-symmetry of $R$ implies that $\nabla^c$ is a G$_2$-instanton. This has been observed for example in \cite{Harland:2011zs}.
\end{rem}
We recommend \cite{Agricola:2006} for a deeper introduction to connections with skew-torsion, particularly in relation to string theory.

\section{Contact Geometry}
\label{sec:Contact}

In this section we recall some standard definitions of contact geometry that will be used in the paper and we set our notations. A standard reference on the topic is \cite{Boyer:2007book}.
\begin{defi}
    An \emph{almost contact structure} $(\xi,\eta,\phi)$ on a $2m+1$-dimensional smooth manifold $M$ consists of a vector field $\xi$, a one-form $\eta$ and a $(1,1)$-tensor field $\phi$ satisfying
    \begin{equation*}
        \phi^2=-\operatorname{Id}+\eta\otimes\xi \, , \qquad \eta(\xi)=1 \, .
    \end{equation*}
    In this case, the tangent bundle of $M$ splits as $TM=\mathcal{H}\oplus\langle\xi\rangle$, where $\mathcal{H}=\ker(\eta)$ is a $2m$-dimensional distribution. We call $\xi$ the \emph{Reeb vector field} and we say that $(M,\xi,\eta,\phi)$ is an \emph{almost contact manifold}.
\end{defi}

\begin{rem}
    Every almost contact manifold admits a Riemannian metric $g$ which is \emph{compatible} with the almost contact structure in the following sense:
    \begin{equation*}
        g(\phi X,\phi Y)=g(X,Y)-\eta(X)\eta(Y) \, .
    \end{equation*}
    In this situation, $\eta=g(\cdot,\xi)$ and $\mathcal{H}=\langle\xi\rangle^\perp$. We say that $(\xi,\eta,\phi,g)$ is an \emph{almost contact metric structure} and that $(M,\xi,\eta,\phi,g)$ is an \emph{almost contact metric manifold}.
\end{rem}

\begin{rem}
\label{rem:acsisUm}
    A choice of almost contact metric structure on a manifold $M$ is equivalent to a choice of $\mathrm{U}(m)$-structure on $M$---that is, a reduction of the structure group to $\mathrm{U}(m)\times 1$.
    
    In particular, note that $\phi\vert_\mathcal{H}$ defines an almost complex structure on $\mathcal{H}$. Furthermore, the \emph{fundamental form} $\Phi(X,Y)=g(X,\phi Y)$ defines a hermitian form on $\mathcal{H}$ via $\Phi\vert_\mathcal{H}$. 
\end{rem}

\begin{defi}
    An \emph{$\alpha$-contact metric structure} on a $2m+1$-dimensional smooth manifold $M$ is an almost contact metric structure such that
    \begin{equation*}
        \dd\eta=2\alpha\Phi \, ,
    \end{equation*}
    where $\alpha\in\mathbb{R}\backslash \lbrace 0\rbrace$ and $\Phi(X,Y)=g(X,\phi Y)$ is the {fundamental form}. Note this is a particular case of a \emph{contact} structure since $\eta\wedge(\dd\eta)^m\neq 0$, and we call $\eta$ the \emph{contact form}. We say that $M$ is an \emph{$\alpha$-contact metric manifold}.
\end{defi}

\begin{defi}
	An \emph{$\alpha$-Sasaki structure} on a $2m+1$-dimensional smooth manifold $M$ is an $\alpha$-contact metric structure which is \emph{normal}, meaning that the following quantity vanishes for all vector fields $X,Y$:
    \begin{equation*}
        N_\phi(X,Y)=[\phi X,\phi Y]+\phi^2[X,Y]-\phi [\phi X, Y]-\phi [ X,\phi Y] +\dd\eta (X,Y)\xi \, .
    \end{equation*}
    In this case we call $M$ an \emph{$\alpha$-Sasaki manifold}.
\end{defi}
The vanishing of $N_\phi$ makes $\phi$ an integrable transverse structure on $\mathcal{H}$, see e.g. \cite{Boyer:2007book}. In other words
\begin{thm}
\label{thm:sasakiprojection}
    An $\alpha$-Sasaki structure admits a locally defined Riemannian submersion $\pi\colon (M,g)\to (N,g_N)$ along $\mathcal{V}\coloneqq \langle\xi\rangle$ such that $\phi$ projects under $\pi$ and induces a Kähler structure on $N$.
\end{thm}

We will be especially interested in $\alpha$-Sasaki manifolds of a particular type:

\begin{defi}[\cite{Boyer:2004eh}]
\label{def:etaEinstein}
    An $\alpha$-Sasaki structure on a manifold $M$ is \emph{$\eta$-Einstein} if there are constants $\lambda$, $\nu$ such that $\operatorname{Ric}_g=\lambda g +\nu \, \eta\otimes\eta$. In this case we call $M$ an \emph{$\eta$-Einstein $\alpha$-Sasaki manifold}.
\end{defi}

\begin{rem}
    As shown in \cite{Boyer:2004eh}, every simply connected $\eta$-Einstein $\alpha$-Sasaki manifold $M$ is spin. In \Cref{sec:su3case}, we will characterize these manifolds in terms of the existence of a particular $\mathrm{SU}(m)$-structure that constitutes a reduction from the $\mathrm{U}(m)$-structure associated to the metric almost contact structure.
\end{rem}

\begin{defi}
\label{def:SUmstructure}
     An \emph{$\mathrm{SU}(m)$-structure} $(\xi,\eta,\phi,g,\Phi,\Omega)$ on a $2m+1$-dimensional manifold $M$ is given by an almost contact metric structure $(\xi,\eta,\phi,g)$ with fundamental form $\Phi$ together with a nowhere-vanishing complex $m$-form $\Omega$. The form $\Omega$ is of type $(m,0)$ with respect to the complex structure on $\mathcal{H}$ induced by $\phi$ and satisfies
	\begin{equation}
        \label{eq:normalizationSUmstructure}
		\frac{1}{m!}\,\Phi^m=\frac{(-1)^{m(m-1)/2}\, i^m}{2^m}\,\Omega\wedge\bar{\Omega} \, .
	\end{equation}
\end{defi}

\begin{rem}
    This is in fact equivalent to the usual definition of an $\mathrm{SU}(m)$-structure on $M$ as a reduction of the structure group to $\mathrm{SU}(m)\times 1$. By \cite{Conti:2008}, the intrinsic torsion of the structure is completely determined by the exterior derivatives of the forms $\Phi$ and $\Omega$.
\end{rem}

We will call a frame \emph{adapted} to a $G$-structure if it is a section of the principal bundle associated to that $G$-structure. 
For $\mathrm{SU}(m)$-structures an orthonormal frame $\{e_1,\dots,e_{2m+1}\}$ is adapted if the dual coframe $\lbrace e^\mu \rbrace_{\mu=1}^{2m+1}$ satisfies
\begin{equation}
\label{eq:explicitSU3forms}
    \eta=e^{1}\, , \qquad \Phi=-\sum_{a=1}^m e^{2a}\wedge e^{2a+1} \, , \qquad \Omega=(e^{3}+i\, e^{2})\wedge\cdots\wedge (e^{2m+1}+i\, e^{2m}) \, .
\end{equation}
We will be particularly interested in the situation where we have three different almost contact structures interacting with each other. As it turns out these admit just the right amount of flexibility for finding solutions.

\begin{defi}
\label{def:ac3s}
    An \emph{almost $3$-contact structure} on a $4n+3$-dimensional smooth manifold $M$ consists of three almost contact structures $(\xi_i,\eta_i,\phi_i)_{i=1,2,3}$ that satisfy the compatibility conditions
    \begin{equation}
    \label{eq:a3csproperties}
        \phi_i\circ\phi_j=\phi_k-\eta_j\otimes\xi_i\, ,\qquad \eta_i\circ\phi_j=\eta_k\, .
    \end{equation}
    Here and from hereon we employ the convention that $(ijk)$ is an arbitrary even permutation of $(123)$.
\end{defi}

\begin{defi}
\label{rem:acm3s}
    Every almost 3-contact structure admits a Riemannian metric $g$ which is \emph{compatible} with each of the three almost contact structures.
     We then call $(\xi_i,\eta_i,\phi_i,g)_{i=1,2,3}$ an \emph{almost 3-contact metric structure} and $(M,\xi_i,\eta_i,\phi_i,g)_{i=1,2,3}$ an \emph{almost 3-contact metric manifold}.
     
     We denote the vertical and horizontal spaces of $(M,\xi_i,\eta_i,\phi_i,g)_{i=1,2,3}$ as $\mathcal{V}\coloneqq \langle\xi_1,\xi_2,\xi_3\rangle$ and $\mathcal{H}\coloneqq\mathcal{V}^\perp=\bigcap_{i=1}^3\ker\eta_i\,$.
\end{defi}

\begin{rem}
\label{rem:acm3sisSpm}
    A choice of almost 3-contact metric structure on a manifold $M$ is equivalent to a choice of $\mathrm{Sp}(n)$-structure on $M$---that is, a reduction of the structure group to the group $\mathrm{Sp}(n)\times 1_3$.

    In particular, we can interpret $\{\phi_i\}_{i=1}^3$ as an almost hyperhermitian structure on $\mathcal{H}$ with fundamental forms $\Phi_i^\mathcal{H}\coloneqq\Phi_i\vert_\mathcal{H}$.
\end{rem}

We also have a notion of adapted frame for $\mathrm{Sp}(n)$-structures: we say an orthonormal frame $\{e_1,e_2,\dots,e_{4m+3}\}$ is adapted if the dual coframe $\lbrace e^\mu \rbrace_{\mu=1}^{4m+3}$ satisfies
\begin{align}
\begin{split}
\label{eq:explicitSpnforms}
    \eta_i&=e^{i}\, , \qquad \text{for } i\in\lbrace 1, 2, 3 \rbrace \, , \\
    \Phi_1&= -\eta_{23}-\sum_{r=1}^n \left( e^{4r+4}\wedge e^{4r+5} + e^{4r+6}\wedge e^{4r+7} \right) \, , \\ 
    \Phi_2&= -\eta_{31}-\sum_{r=1}^n \left( e^{4r+4}\wedge e^{4r+6} - e^{4r+5}\wedge e^{4r+7} \right) \, , \\
    \Phi_3&= -\eta_{12}-\sum_{r=1}^n \left( e^{4r+4}\wedge e^{4r+7} + e^{4r+5}\wedge e^{4r+6} \right) \, ,
\end{split}
\end{align}
where $\eta_{ij}=\eta_i\wedge\eta_j$. Such a frame can always be constructed using the properties \eqref{eq:a3csproperties}.

\section{Characteristic holonomy Sp(1)Sp(1)}
\label{sec:su2case}

In this section we study the heterotic G$_2$ system on 3-$(\alpha,\delta)$-Sasaki manifolds. We begin with a description of these manifolds and their most relevant properties before discussing a 1-parameter family of connections $\nabla^\lambda$ on the tangent bundle. We then specialize our results to dimension $7$ and solve the heterotic G$_2$ system using the family $\nabla^\lambda$.

\subsection{3-$(\alpha,\delta)$-Sasaki manifolds}

We begin by introducing 3-$(\alpha,\delta)$-Sasaki manifolds and some of their main properties, we refer the reader to \cite{Agricola:2018, Agricola:2021, Agricola:2023} for further details.
\begin{defi}
    A \emph{$3$-$(\alpha,\delta)$-Sasaki manifold} is an almost $3$-contact metric manifold satisfying the differential condition
\begin{equation}
    \mathrm{d}\eta_i=2\alpha\Phi_i+2(\alpha-\delta)\eta_{jk}=2\alpha\Phi_i^{\mathcal{H}}-2\delta\eta_{jk}\,,\label{def3ad}
\end{equation}
    where $\alpha$ and $\delta$ are real constants with $\alpha\neq 0$.
\end{defi}

The definition suggests that $\delta=0$ is a special case, and we will see this come into effect later on. In fact, we can classify 3-$(\alpha,\delta)$-Sasaki manifolds in three different types as follows:

\begin{defi}
\label{def:valuesofalphadelta}
    A $3$-$(\alpha,\delta)$-Sasaki manifold is called
    \begin{enumerate}[a)]
        \item \emph{degenerate} if $\delta=0\,$,
        \item \emph{positive} if $\alpha\delta>0\,$,
        \item \emph{negative} if $\alpha\delta<0\,$.
    \end{enumerate}
\end{defi}

These definitions are motivated by the fact that each type of $3$-$(\alpha,\delta)$-Sasaki manifold is invariant under a class of deformations called $\mathcal{H}$-homothetic deformations:
\begin{prop}[\cite{Agricola:2018}]
\label{prop:Hhomotheticdeformations}
Let  $(M,g,\xi_i,\eta_i,\phi_i)_{i=1,2,3}$ be a $3$-$(\alpha,\delta)$-Sasaki manifold. Consider the deformed structure tensors
\begin{equation*}
    \tilde{g}=c^2g|_\mathcal{V}+ag|_\mathcal{H}\, ,\quad \tilde{\eta}_i=c\,\eta_i \, ,\quad \tilde{\xi}_i=\frac1c\xi_i \, ,\quad \tilde{\phi}_i=\phi_i \, .
\end{equation*}
Then $(M,\tilde{g},\tilde{\xi}_i,\tilde{\eta}_i,\tilde{\phi}_i)_{i=1,2,3}$ is a $3$-$(\tilde{\alpha},\tilde{\delta})$-Sasaki manifold with
\begin{equation*}
    \tilde{\alpha}=\frac ca\alpha\, ,\qquad\tilde{\delta}=\frac{\delta}{c}\, .
\end{equation*}
In particular, the deformed structure is degenerate/positive/negative if and only if the initial structure was.
\end{prop}

In addition, we will see in \Cref{thm:submersion3alphadelta} that the parameter $\alpha\delta$ is related to the curvature of the base space of a submersion, providing further motivation for the classification.

A key property of $3$-$(\alpha,\delta)$-Sasaki manifolds is that they admit a particularly well-behaved connection:
\begin{thm}[\cite{Agricola:2018}]
    A $3$-$(\alpha,\delta)$-Sasaki manifold $(M,g,\xi_i,\eta_i,\phi_i)$ admits a unique connection $\nabla$ with skew torsion such that
    \begin{equation*}
        \nabla_X\phi_i=\beta(\eta_k(X)\phi_j-\eta_j(X)\phi_k) \, ,
    \end{equation*}
    for any even permutation $(ijk)$ of $(123)$ and $\beta=2(\delta-2\alpha)$. Its torsion is given by
    \begin{equation}\label{torsion}
        T=2\alpha\sum_{i=1}^3\eta_i\wedge\Phi_i^\mathcal{H}+2(\delta-4\alpha)\eta_{123} \, ,
    \end{equation}
    and satisfies $\nabla T=0$. 
\end{thm}
If $\beta=0$, or equivalently $\delta=2\alpha$, then $\nabla\phi_i=0$ for all $i\in\lbrace1,2,3\rbrace$ and we call the $3$-$(\alpha,\delta)$-Sasaki manifold \emph{parallel}.
\begin{defi}
\label{def:canonicalconnection3alphadelta}
    The connection above is called the \emph{canonical connection}.\footnote{The reader should be aware that in other references, such as \cite{delaOssa:2021qlt} or \cite{Bryant:2005mz}, the name \emph{canonical connection} is used to denote different connections.} We refer to its torsion, curvature and other associated tensors as canonical.
\end{defi}

The canonical connection has holonomy Sp(n)Sp(1), whose representation in dimension $4n+3$ is reducible and gives rise to a locally defined submersion:
\begin{thm}[\cite{Agricola:2021}]
\label{thm:submersion3alphadelta}
    A $3$-$(\alpha,\delta)$-Sasaki manifold admits a locally defined Riemannian submersion $\pi\colon(M,g)\to(N,g_N)$ with totally geodesic fibres tangent to $\mathcal{V}$. The base space $(N,g_N)$ inherits a quaternionic Kähler structure that is locally given by the almost complex structures $J_i=\pi_*\circ\phi_i\circ s_*$ where $s$ is any local section of $\pi$.

    Furthermore, the canonical connection projects to the Levi-Civita connection on the base in the sense that
    \begin{equation*}
        \nabla^{g_N}_XY=\pi_*(\nabla_{\overline{X}} \overline{Y}) \, ,
    \end{equation*}
    where $\overline{X}$ and $\overline{Y}$ are the horizontal lifts of $X,Y$. The scalar curvature of the base space $N$ takes the value $\mathrm{scal}^{g_N}=16n(n+2)\alpha\delta$.
\end{thm}

\begin{rem}
The Reeb vector fields satisfy $[\xi_i,\xi_j]=2\delta\xi_k$, see \cite{Agricola:2018}. Thus, if $\delta\neq 0$ the leaves of the canonical submersion defined above are orbits of an $\mathrm{SU}(2)$ action. As a result, the leaves are necessarily compact and the leaf space $N$ obtains globally the structure of an orbifold.

In the degenerate case $\delta=0$ the situation is different: the leaves are orbits of the action of a 3-dimensional abelian group (e.g. $\R^3$ or $\mathbb{T}^3$) and the leaves are not necessarily compact anymore. Furthermore, in this case the base has a hyperk\"ahler structure. For a more detailed description of the different possibilities we refer the reader to \cite{Goertsches:2022}.
\end{rem}

As degenerate $3$-$(\alpha,\delta)$-Sasaki manifolds are central to our solutions of the heterotic $G_2$-system, we briefly review some explicit examples.

\begin{ex}[\cite{Agricola:2018}]
    The quaternionic Heisenberg group $H^{n,\mathbb{H}}$ is the simply connected Lie group with Lie algebra 
    \begin{align*}
        \mathfrak{h}^{n,\mathbb{H}}=\left\{\begin{pmatrix}0& \overline{X}^t& Z\\0&0&X\\0&0&0\end{pmatrix}\in \mathfrak{gl}(n+2,\mathbb{H}),\quad X\in\mathbb{H}^n,\ Z\in\mathrm{Im}\mathbb{H}\right\}.
    \end{align*}
    Consider $H^{n,\mathbb{H}}$ with its canonical left-invariant metric $g$ induced by the standard scalar product on $\mathfrak{h}^{n,\mathbb{H}}$. Let $\{\xi_i\}_{i=1,2,3}$ be an orthonormal basis of $\mathrm{Im}\mathbb{H}\subset\mathfrak{h}^{n,\mathbb{H}}$, $\{\eta_i\}_{i=1,2,3}$ its dual and $\{\varphi_i\}_{i=1,2,3}$ the endomorphisms acting as unit imaginary quaternions from the left on $\mathbb{H}$ and $\mathrm{Im}\mathbb{H}$. Then $(M,g,\xi_i,\eta_i,\varphi_i)_{i=1,2,3}$ is a degenerate $3$-$(\alpha,\delta)$-Sasaki manifold with $\alpha=1$.

    This example is non-compact but admits a cocompact lattice.
\end{ex}

More generally, degenerate $3$-$(\alpha,\delta)$-Sasaki manifolds are obtained via the following construction:

\begin{prop}[\cite{Stecker:2021}]
\label{thm:constructiondegenerate}
    Let $(N,g_N,J_1,J_2,J_3)$ be a hyperk\"ahler manifold with integer K\"ahler classes $[\omega_i]\in H^2(N,\mathbb{Z})$, $i=1,2,3$. Let $M$ be the $T^3$-bundle obtained as a fibre product of the three Boothby-Wang-bundles associated to $[\omega_i]_{i=1,2,3}$. Then $M$ admits a degenerate $3$-$(\alpha,\delta)$-Sasaki structure with Reeb vector fields tangent to the fibres of $M$.
\end{prop}

The compact quotient of the quaternionic Heisenberg group can be understood as applying \Cref{thm:constructiondegenerate} to the flat hyperk\"ahler torus. In addition, \cite[Lemma 2.2]{Cortes:2023} shows that there exists at least one $K3$-surface satisfying the requirements of \Cref{thm:constructiondegenerate}, providing an additional compact example.

\medskip

It will be convenient in what follows to study $3$-$(\alpha,\delta)$-Sasaki manifolds using a spinorial perspective. The right notion of generalized Killing spinor for $3$-$(\alpha,\delta)$-Sasaki manifolds was introduced in \cite{Agricola:2023spin}:

\begin{defi}[\cite{Agricola:2023spin}]
    A spinor $\Psi$ in a $3$-$(\alpha,\delta)$-Sasaki manifold is said to be $\mathcal{H}$-Killing if it satisfies
\begin{equation}
\label{eq:Hkillingspinor}
    \nabla^g_X\Psi=\frac{\alpha}{2} X\cdot \Psi + \frac{\alpha-\delta}{2}\sum_{\ell=1}^3\eta_\ell(X)\Phi_\ell\cdot\Psi \, \qquad \text{for all } X\in TM \, .
\end{equation}
\end{defi}
As explained in \Cref{app:Spnstructuresandspinors}, an almost 3-contact metric manifold can be equivalently described in terms of six spinors $\lbrace\Psi_{i,\pm}\rbrace_{i=1,2,3}$, where each pair of spinors $\Psi_{i,\pm}$ spans the rank 2 bundle \cite{Friedrich:1990zg}
\begin{equation*}
    E_i=\lbrace \Psi\in\Gamma(\Sigma) \, \vert \, \left( -2\,\phi_i(X) + \xi_i\cdot X - X\cdot\xi_i\right) \cdot\Psi=0 \ \ \text{for all vectors } X \rbrace \, .
\end{equation*}
The six spinors span together the bundle $E=E_1+E_2+E_3$, which might not be a direct sum. This means that the spinors might not be linearly independent, as can be seen from particular examples in  \cite{Agricola:2022, Hofmann:2022}. We will later see that this is the situation in dimension 7.

It was shown in \cite{Agricola:2023spin} that for every (simply connected) $3$-$(\alpha,\delta)$-Sasaki manifold the spinors in the bundle $E$ are $\mathcal{H}$-Killing. This actually characterizes $3$-$(\alpha,\delta)$-Sasaki manifolds in the sense of the following proposition:

\begin{prop}
    Let $M$ be an almost 3-contact metric manifold and suppose the spinors $\lbrace\Psi_{i,\pm}\rbrace_{i=1,2,3}$ in the bundle $E$ satisfy \eqref{eq:Hkillingspinor} for some real constants $\alpha$ and $\delta$ with $\alpha\neq 0$. Then, $M$ is $3$-$(\alpha,\delta)$-Sasaki.
\end{prop}
\begin{proof}
    In order to show that \eqref{def3ad} holds, we first need to compute $\nabla^g_X\xi_i\,$. The fact that the spinors $\Psi_{i,\pm}$ satisfy \eqref{eq:oneformsactingonspinorsplusminus} will prove crucial for this. Taking the covariant derivative of the expression $\xi_i\cdot\Psi_{i,+}=\Psi_{i,-}$ and using the Leibniz rule together with the assumption that the spinors satisfy \eqref{eq:Hkillingspinor}, we obtain for any $X\in TM$
    \begin{equation*}
        \left(\nabla^g_X\xi_i\right)\cdot\Psi_{i,+}=\frac{\alpha}{2}\left( X\cdot\xi_i-\xi_i\cdot X \right)\cdot \Psi_{i,+} + \frac{\alpha-\delta}{2}\sum_{\ell=1}^3\eta_\ell(X)\left(\Phi_\ell\cdot\xi_i - \xi_i\cdot\Phi_\ell  \right)\cdot\Psi_{i,+} \, ,
    \end{equation*}
    where we have used \eqref{eq:oneformsactingonspinorsplusminus} again to rewrite $\Psi_{i,-}$ in terms of $\Psi_{i,+}$. Using \eqref{eq:horizontalvectoronpsiplus} and noting that $\Phi_j\cdot\xi_i\cdot\Psi= \left(\xi_i\cdot\Phi_j + 2\,\xi_k\right)\cdot\Psi$ for any spinor $\Psi$, we can compute
        \begin{align*}
        \left(\nabla^g_j\xi_i\right)\cdot\Psi_{i,+}&=-\alpha\,\xi_i\cdot \xi_j \cdot \Psi_{i,+} + \frac{\alpha-\delta}{2}\left(\Phi_j\cdot\xi_i - \xi_i\cdot\Phi_j  \right)\cdot\Psi_{i,+} \\
        &=-\alpha\,\phi_i( \xi_j) \cdot \Psi_{i,+} + \left(\alpha-\delta\right)\xi_k\cdot\Psi_{i,+}=-\delta\,\xi_k\cdot\Psi_{i,+}\, ,
    \end{align*}
    and the non-degeneracy of the Clifford product implies that $\nabla^g_j\xi_i=-\delta\,\xi_k \,$. Analogously, one obtains
    \begin{equation*}
        \nabla^g_i\xi_i=0 \, , \qquad \nabla^g_k\xi_i=\delta\,\xi_j \, , \qquad \nabla^g_X\xi_i=-\alpha\,\phi_i(X) \ \ \text{for } X\in\mathcal{H} \, .
    \end{equation*}
    Since the connection is metric, we immediately obtain $\nabla^g_X\eta_i\,$ from these formulas. In particular, note that with our conventions for the fundamental form  we have $\nabla^g_X\eta_i=\alpha\,\Phi_i(X,\cdot)$ for $X\in\mathcal{H}$. We can then compute in an adapted frame:
    \begin{equation*}
        \dd\eta_i =\sum_{\mu=1}^{4n+3}e^\mu\wedge\nabla^g_{e_\mu}\eta_i=-\delta\,\eta_j\wedge\eta_k + \delta\,\eta_k\wedge\eta_j +\alpha\sum_{r=4}^{4n+3}e^r\wedge \Phi_i(e_r,\cdot)
        =-2\,\delta\,\eta_{jk} +2\,\alpha \Phi^{\mathcal{H}} \, ,
    \end{equation*}
    which shows the result.
\end{proof}

This illustrates that the class of manifolds we are interested in can be equivalently described in terms of spinors. In the case of dimension 7, these spinors take a very particular form that we will describe in detail later.

\subsection{The family of connections $\nabla^\lambda$}

We want to introduce a more general family of connections. We define the three-form
\begin{equation}
\label{eq:G2structure3alphadelta}
\varphi=\eta_{123}+\sum_{i=1}^3\eta_i\wedge\Phi_i^\mathcal{H}\, .
\end{equation}
This is well-defined in all dimensions, but it is especially important in dimension 7 as then $\varphi$ constitutes a G$_2$-structure, compare \cite{Agricola:2018}.

\begin{prop}\label{lambdaconn}
Let $(M,g,\xi_i,\eta_i,\phi_i)_{i=1,2,3}$ be a $3$-$(\alpha,\delta)$-Sasaki manifold. Denote by $\nabla$ its  canonical connection with skew-torsion $T$ and $\varphi$ as above. Then, there exists a family of metric connections $\nabla^\lambda$, $\lambda\in \R$, with parallel torsion $T^\lambda$ that preserve the splitting $TM=\mathcal{V}\oplus\mathcal{H}$, satisfy $\nabla\varphi=0$ and
\begin{equation}\label{deflambdaconn}
\nabla^\lambda_Y\phi_i=(\beta+\lambda)(\eta_k(Y)\phi_j-\eta_j(Y)\phi_k)\, .
\end{equation}
The torsion $T^\lambda(X,Y,Z)\coloneqq g(X,T^\lambda(Y,Z))$ has non-zero components
\begin{align}\label{lambdatorsion}
T^\lambda(X,Y,Z)=
\begin{cases}T(X,Y,Z)+2\lambda \varphi(X,Y,Z) & X,Y,Z\in\mathcal{V}\, ,\\
T(X,Y,Z) &  X\in \mathcal{V},\ Y,Z\in\mathcal{H}\, ,\\
T(X,Y,Z)-\frac\lambda2\varphi(X,Y,Z) & Y\in \mathcal{V},\ X,Z\in\mathcal{H}\, ,\\
T(X,Y,Z)-\frac\lambda2\varphi(X,Y,Z) & Z\in \mathcal{V},\ X,Y\in\mathcal{H}\, .\end{cases}
\end{align}
These connections are projectable under the canonical submersion $\pi\colon(M,g)\to(N,g_N)$ in the sense that
\begin{equation}\label{lambdasubm}
\nabla^{g_N}_XY=\pi_*(\nabla^\lambda_{\overline{X}}\overline{Y})\, .
\end{equation}
\end{prop}
\begin{proof}
We obtain $\nabla^\lambda$ as the unique metric connection with torsion $T^\lambda$ or equivalently by setting $\nabla^\lambda=\nabla+\Delta^\lambda$ where
\begin{align}
\Delta^\lambda(X,Y,Z)&=\frac 12(T^\lambda(X,Y,Z)-T^\lambda(Y,Z,X)+T^\lambda(Z,X,Y)-T(X,Y,Z))\notag\\
&=\begin{cases}\lambda\varphi(X,Y,Z) & X,Y,Z\in\mathcal{V}\, ,\\ -\frac\lambda2\varphi(X,Y,Z) & Y\in \mathcal{V}, X,Z\in \mathcal{H}\, , \\0& \text{else}\, . \end{cases}\label{Deltalambda}
\end{align}
We remark that $g(X,\nabla^\lambda_YZ)=g(X,\nabla_YZ)$ when $Y\in\mathcal{H}$. This immediately proves \eqref{lambdasubm} from the corresponding statement for $\nabla$.

We now prove \eqref{deflambdaconn}. If $Y\in\mathcal{H}$ then as above $\nabla^\lambda_Y\phi_i=\nabla_Y\phi_i=0$ so \eqref{deflambdaconn} holds trivially. Now suppose $Y\in \mathcal{V}$:
\begin{align*}
g(X,(\nabla^\lambda_Y\phi_i)Z)&=g(X,\nabla^\lambda_Y(\phi_iZ))-g(X,\phi_i\nabla^\lambda_YZ)\\
&=g(X,(\nabla_Y\phi_i)Z)+\Delta^\lambda(X,Y,\phi_iZ)+\Delta^\lambda(\phi_iX,Y,Z)\\
&=\beta g(X,(\eta_k(Y)\phi_j-\eta_j(Y)\phi_k)Z)+\Delta^\lambda(X,Y,\phi_iZ)+\Delta^\lambda(\phi_iX,Y,Z)\, .
\end{align*}
When $X,Z$ are of different type, $\Delta^\lambda(X,Y,Z)$ and $\Phi_i(Z,X)=g(Z,\phi_iX)$ both vanish. Therefore, we only need to check for $X,Z$ which are both in either $\mathcal{V}$ or $\mathcal{H}$. If $X,Z\in \mathcal{H}$
\begin{align*}
\Delta^\lambda(X,Y,\phi_iZ)+\Delta^\lambda(\phi_iX,Y,Z)&=\frac \lambda 2\sum_{\ell=1}^3\eta_\ell(Y)(\Phi_\ell(X,\phi_iZ)+\Phi_\ell(\phi_iX,Z))\\
&=-\lambda (\eta_j(Y)g(X,\phi_kZ)-\eta_k(Y)g(X,\phi_jZ)) \, ,
\end{align*}
where we have used \eqref{eq:a3csproperties}. If $X,Z\in \mathcal{V}$ then
\begin{align*}
\Delta^\lambda(X,Y,\phi_iZ)&+\Delta^\lambda(\phi_iX,Y,Z)=\lambda(\eta_{123}(X,Y,\phi_iZ)+\eta_{123}(\phi_iX,Y,Z))\\
&=\lambda(\eta_k(Y)(\eta_k(Z)\eta_i(X)-\eta_i(Z)\eta_k(X))-\eta_j(Y)(\eta_i(Z)\eta_j(X)-\eta_j(Z)\eta_i(X)))\\
&=\lambda g(X,\eta_k(Y)\phi_jZ-\eta_j(Y)\phi_kZ)\, .
\end{align*}
Thus, we have shown \eqref{deflambdaconn}. We then have that $\nabla^\lambda$ preserves $\mathcal{V}$, as $\phi_i(\nabla^\lambda\xi_i)=-(\nabla^\lambda\phi_i)\xi_i\in\mathcal{V}$, and consequently the orthogonal splitting $TM=\mathcal{V}\oplus\mathcal{H}$.

To be more precise we have
\begin{equation}\label{nablaxi}
\nabla^\lambda_X\xi_i=-\phi_i^2(\nabla^\lambda_X\xi_i)=\phi_i(\nabla^\lambda_X\phi_i)\xi_i=(\beta+\lambda)(\eta_k(X)\xi_j-\eta_j(X)\xi_k)\, ,
\end{equation}
where we have used that $0=d(g(\xi_i,\xi_i))=2\eta_i(\nabla^\lambda\xi_i)$.

It remains to show $\nabla T^\lambda=\nabla\varphi=0$. Observe that $\nabla^\lambda$ preserves a given tensor if and only if it preserves its components with respect to the splitting $TM=\mathcal{V}\oplus\mathcal{H}$. In particular, from \eqref{torsion}, \eqref{eq:G2structure3alphadelta} and \eqref{lambdatorsion} we see that $\nabla^\lambda T^\lambda=\nabla^\lambda\varphi=\nabla^\lambda T=0$ if and only if $\nabla^\lambda\sum_{i=1}^3\eta_i\wedge\Phi_i^\mathcal{H}=\nabla^\lambda\eta_{123}=0$. The latter is just the volume form in $\mathcal{V}$ and thus parallel. Using \eqref{nablaxi} and \eqref{deflambdaconn} we find
\begin{align*}
\nabla^\lambda_X(\eta_i\wedge \Phi_i)&=(\nabla^\lambda_X\eta_i)\wedge\Phi_i+\eta_i\wedge(\nabla^\lambda_X\Phi_i)\\
&=(\beta+\lambda)(\eta_k(X)\eta_j\wedge\Phi_i-\eta_j(X)\eta_k\wedge\Phi_i+\eta_k(X)\eta_i\wedge\Phi_j-\eta_j(X)\eta_i\wedge\Phi_k)\, .
\end{align*}
If we sum over all $i\in\lbrace 1,2,3 \rbrace$ we see that all terms cancel out, and $\eta_i\wedge \Phi_i=\eta_i\wedge \Phi_i^\mathcal{H}-\eta_{123}$ yields the result.
\end{proof}

In what follows, we will denote all tensors corresponding to $\nabla^\lambda$ by a $\lambda$ index, such as for example $T^\lambda$ and $R^\lambda$.

\begin{rem}
\label{rem:connectionasadeformation}
In dimension $7$ the family agrees with two other constructions of G$_2$-connections for certain values of the parameter $\lambda$:\\
If $\lambda<4\alpha$ then $\nabla^\lambda$ is the canonical connection of the $3$-$(\alpha,\delta)$-Sasaki structure obtained by the $\mathcal{H}$-homothetic deformation with parameters
\begin{equation*}
a=1\, ,\qquad c=\sqrt{1-\frac\lambda{4\alpha}}\, .
\end{equation*}
Since the deformation changes the metric, $\nabla^\lambda$ will not have skew-torsion with respect to the original metric $g$. However, we still have $\nabla^\lambda g=0$ and $\nabla^\lambda\varphi=0$.\\
If $\delta\neq 5\alpha$, then $\nabla^\lambda$ is given as the $1$-parameter family of connections compatible with the G$_2$-structure \eqref{eq:G2structure3alphadelta}, compare \cite{Bryant:2005mz}. In this case $\lambda=(1+2a)(5\alpha-\delta)$, where $a$ is the parameter of the family as written in \cite{delaOssa:2021qlt}.
\end{rem}

\begin{lem}\label{nablavert}
If $Z$ is a horizonal lift of a vector field on $N$, $Y\in \mathcal{V}$ and $X\in\mathcal{H}$ then
\begin{equation*}
g(X,\nabla^\lambda_YZ)=\bigg(2\alpha-\frac\lambda2\bigg)\sum_{i=1}^3\eta_i(Y)\Phi_i(Z,X)\, .
\end{equation*}
\end{lem}

\begin{proof}
For a combination of vector fields $X,Y,Z$ as above the canonical connection satisfies $(\nabla_YZ)_\mathcal{H}=-2\alpha\sum_{i=1}^3\eta_i(Y)\phi_iZ$, compare \cite[Lemma 2.2.1]{Agricola:2021}. Using the notation in the proof of \Cref{lambdaconn} we find
\begin{align*}
g(X,\nabla^\lambda_YZ)&=g(X,\nabla_YZ)+\Delta^\lambda(X,Y,Z)=2\alpha\sum_{i=1}^3\eta_i(Y)\Phi_i(Z,X)-\frac\lambda2\varphi(X,Y,Z)\\
&=\bigg(2\alpha-\frac\lambda2\bigg)\sum_{i=1}^3\eta_i(Y)\Phi_i(Z,X)\, .\qedhere
\end{align*}
\end{proof}

\begin{lem}\label{projcurvlemma}
For tangent vectors $X,Y,Z,V\in TN$ with horizontal lifts $\overline{X},\overline{Y},\overline{Z},\overline{V}$ we have
\begin{equation}\label{projcurv}
R^{g_N}(X,Y,Z,V)=R^\lambda(\overline{X},\overline{Y},\overline{Z},\overline{V})-\alpha(4\alpha-\lambda)\sum_{i=1}^3\Phi_i(\overline{X},\overline{Y})\Phi_i(\overline{Z},\overline{V})\, .
\end{equation}
In particular, $R^\lambda|_{\Lambda^2\mathcal{H}\otimes\Lambda^2\mathcal{H}}$ is pairwise symmetric.
\end{lem}

\begin{proof}
We recall from \cite{Agricola:2021} that for $X,Y\in\mathcal{H}$
\begin{equation*}
[X,Y]_{\mathcal{V}}=-2\alpha\sum_{i=1}^3\Phi_i(X,Y)\xi_i\, .
\end{equation*}
Using this as well as the identities \eqref{lambdasubm} and \Cref{nablavert} we find
\begin{align*}
g_N(\nabla^{g_N}_X\nabla^{g_N}_Y Z,V)&=g_N(\nabla^{g_N}_X\pi_*(\nabla^\lambda_{\overline{Y}}\overline{Z}),V)
=g(\nabla^\lambda_{\overline{X}}\overline{\pi_*(\nabla^\lambda_{\overline{Y}}\overline{Z})},\overline{V})
=g(\nabla^\lambda_{\overline{X}}\nabla^\lambda_{\overline{Y}}\overline{Z},\overline{V})\, ,\\
g_N(\nabla^{g_N}_{[X,Y]}Z,V)&=g_N(\pi_*\nabla^\lambda_{\overline{[X,Y]}}\overline{Z},V)=g(\nabla^\lambda_{[\overline{X},\overline{Y}]}\overline{Z},\overline{V})-g(\nabla^\lambda_{[\overline{X},\overline{Y}]_{\mathcal{V}}}\overline{Z},\overline{V})\\
&=g(\nabla^\lambda_{[\overline{X},\overline{Y}]}\overline{Z},\overline{V})-\left(2\alpha-\frac\lambda2\right)\sum_{i=1}^{3}\eta_i([\overline{X},\overline{Y}]_{\mathcal{V}})\Phi_i(\overline{Z},\overline{V})\\
&=g(\nabla^\lambda_{[\overline{X},\overline{Y}]}\overline{Z},\overline{V})+\alpha(4\alpha-\lambda)\sum_{i=1}^{3}\Phi_i(\overline{X},\overline{Y})\Phi_i(\overline{Z},\overline{V})\, .
\end{align*}
Combining both identities we obtain \eqref{projcurv}.
\end{proof}
Observe that $\alpha(4\alpha-\lambda)=-\alpha(\beta+\lambda)+2\alpha\delta$. As we will soon see these summands relate to the geometry of the fibres and base of the canonical submersion, respectively.

\begin{lem}
We have for $X,Y,Z\in TM$
\begin{align}\label{CommRvarphi}
\begin{split}
R^\lambda(X,Y)\phi_iZ-\phi_iR^\lambda(X,Y)Z&=2\alpha(\beta+\lambda)(\Phi_k^\mathcal{H}(X,Y)\phi_jZ-\Phi_j^\mathcal{H}(X,Y)\phi_kZ)\\
&\qquad -(\beta+\lambda)(4\alpha-\lambda)(\eta_{ij}(X,Y)\phi_jZ-\eta_{ki}(X,Y)\phi_kZ)\, ,
\end{split}
\end{align}
and for $i\in\lbrace 1,2,3 \rbrace$
\begin{align}\label{Ronxi}
\begin{split}
R^\lambda(X,Y)\xi_i&=2\alpha(\beta+\lambda)(\Phi_k^\mathcal{H}(X,Y)\xi_j-\Phi_j^\mathcal{H}(X,Y)\xi_k)\\
&\qquad -(\beta+\lambda)(4\alpha-\lambda)(\eta_{ij}(X,Y)\xi_j-\eta_{ki}(X,Y)\xi_k)\, .
\end{split}
\end{align}
Consequently we obtain
\begin{align}\label{nablacurv}
R^\lambda(X,Y,Z,V)&=-(\beta+\lambda)\left(\sum_{i=1}^3\Phi_i(X,Y)\Phi_i(Z,V)\right)\begin{cases}4\alpha-\lambda & X,Y,Z,V\in \mathcal{V}\, ,\\
2\alpha & X,Y\in \mathcal{H},\ Z,V\in \mathcal{V}\, ,\\
2\alpha-\frac\lambda2 & X,Y\in\mathcal{V},\ Z,V\in \mathcal{H}\, .
\end{cases}
\end{align}
In addition, $R^\lambda(X,Y,Z,V)$ vanishes if $X,Y$ or $Z,V$ are of mixed type. Finally, for $X,Y,Z\in \mathcal{H}$ and $i\in\lbrace 1,2,3 \rbrace$
\begin{equation}\label{RonPhi}
R^\lambda(X,Y,Z,\phi_iZ)+R^\lambda(X,Y,\phi_jZ,\phi_kZ)=2\alpha(\beta+\lambda)\Phi_i(X,Y)\|Z\|^2\, .
\end{equation}
\end{lem}

\begin{proof}
We first show \eqref{CommRvarphi}. We compute in analogy to \cite{Agricola:2023}:
\begin{align*}
R^\lambda(&X,Y)\phi_i Z-\phi_iR^\lambda(X,Y) Z=(\nabla^\lambda_X(\nabla^\lambda_Y\phi_i))Z-(\nabla^\lambda_Y(\nabla^\lambda_X\phi_i))Z-(\nabla^\lambda_{[X,Y]}\phi_i)Z\\
&=(\beta+\lambda)(X(\eta_k(Y))\phi_j+\eta_k(Y)(\nabla^\lambda_X\phi_j)-X(\eta_j(Y))\phi_k-\eta_j(Y)(\nabla^\lambda_X\phi_k)) Z\\
&\quad -(\beta+\lambda)(Y(\eta_k(X))\phi_j+\eta_k(X)(\nabla^\lambda_Y\phi_j)-Y(\eta_j(X))\phi_k-\eta_j(X)(\nabla^\lambda_Y\phi_k)) Z\\
&\quad -(\beta+\lambda)(\eta_k([X,Y])\phi_j-\eta_j([X,Y])\phi_k) Z\\
&=(\beta+\lambda)(X(\eta_k(Y))-Y(\eta_k(X))-\eta_k([X,Y]))\phi_j Z\\
&\qquad\qquad\qquad -(X(\eta_j(Y))-Y(\eta_j(X))-\eta_j([X,Y]))\phi_k Z)\\
&\qquad +(\beta+\lambda)^2((\eta_k(Y)\eta_i(X)-\eta_k(X)\eta_i(Y))\phi_k Z+(\eta_j(Y)\eta_i(X)-\eta_j(X)\eta_i(Y))\phi_j Z)\\
&=(\beta+\lambda)(\mathrm{d}\eta_k(X,Y)\phi_j Z-\mathrm{d}\eta_j(X,Y)\phi_k Z)- (\beta+\lambda)^2(\eta_{ki}(X,Y)\phi_k Z-\eta_{ij}(X,Y)\phi_j Z)\\
&=(\beta+\lambda)\big(2\alpha(\Phi_k^\mathcal{H}(X,Y)\phi_j Z-\Phi_j^\mathcal{H}(X,Y)\phi_k Z)\\
&\qquad+(4\alpha-\lambda)(\eta_{ki}(X,Y)\phi_k Z-\eta_{ij}(X,Y)\phi_j Z)\big)\, ,
\end{align*}
where in the last step we have used \eqref{def3ad} and $\beta+\lambda=2\delta-(4\alpha-\lambda)$.

To obtain \eqref{Ronxi}, we let \eqref{CommRvarphi} act on $\xi_i$ and apply $\phi_i$ from the left. Analogously, to show \eqref{RonPhi} it is enough to consider \eqref{CommRvarphi} for the tensor field $\phi_j$ and contract it with $\phi_kZ$.

The first and second lines of \eqref{nablacurv} are immediate from \eqref{Ronxi}. For the final line, we may extend $Z$ as a horizontal lift and choose the vector fields $X,Y$ in such a way that $\eta_i(X),\eta_i(Y)$ are constant for every $i\in\lbrace 1,2,3 \rbrace$. Then we use \Cref{nablavert} to obtain
\begin{align*}
R^\lambda(X,Y,&Z,V)=g((\nabla^\lambda_X\nabla^\lambda_Y-\nabla^\lambda_Y\nabla^\lambda_X-\nabla^\lambda_{[X,Y]})Z,V)\\
&=-\frac12(4\alpha-\lambda)\sum_{i=1}^3g(\eta_i(Y)\nabla^\lambda_X(\phi_iZ)-\eta_i(X)\nabla^\lambda_Y(\phi_iZ)
-\eta_i([X,Y])\phi_iZ,V) \, .
\end{align*}
Observe that while $Z$ is projectable $\phi_iZ$ is not. Nevertheless, we have
\begin{align*}
\nabla^\lambda_Y(\phi_iZ)&=(\nabla^\lambda_Y\phi_i)Z+\phi_i(\nabla^\lambda_YZ)\\
&=(\beta+\lambda)(\eta_k(Y)\phi_j(Z)-\eta_j(Y)\phi_k(Z))-\frac{4\alpha-\lambda}2\sum_{\ell=1}^3\eta_\ell(Y)\phi_i\phi_\ell Z\\
&=(\beta+2\alpha+\lambda-\frac 12\lambda)(\eta_k(Y)\phi_jZ-\eta_j(Y)\phi_kZ)+\frac 12(4\alpha-\lambda)\eta_i(Y)Z \, ,
\end{align*}
where we can rewrite $\beta+2\alpha=2\delta-2\alpha$. We can now insert this formula in the first two terms of the expression of $R^\lambda(X,Y,Z,V)$, and use that $[\xi_i,\xi_j]=2\delta\xi_k$ in the last one to obtain
\begin{align*}
R^\lambda(X,Y,Z,V)
&=-\frac12(4\alpha-\lambda)(4\delta-4\alpha+\lambda)\cyclic{i,j,k}\eta_{ij}(X,Y)g(\phi_kZ,V)\\
&\qquad\qquad -\delta(4\alpha-\lambda)\cyclic{i,j,k}\eta_{jk}(X,Y)\Phi_i(Z,V)\\
&=-\frac 12(4\alpha-\lambda)(4\delta-4\alpha+\lambda)\sum_{i=1}^3\Phi_i(X,Y)\Phi_i(Z,V)\\
&\qquad\qquad+\delta(4\alpha-\lambda)\sum_{i=1}^3\Phi_i(X,Y)\Phi_i(Z,V)\\
&=-\frac 12(4\alpha-\lambda)(\beta+\lambda)\sum_{i=1}^3\Phi_i(X,Y)\Phi_i(Z,V)\, .
\end{align*}
The final statement we need to prove is that $R(X,Y,Z,V)$ vanishes if $X,Y$ or $Z,V$ are of different type. By \eqref{Ronxi}, the statement is clear for $Z,V$. Again using \eqref{Ronxi}, we may assume that $X\in \mathcal{V}$ and $Y,Z,V\in \mathcal{H}$. Using that $\nabla^\lambda\Delta^\lambda=0$ and the Bianchi identity for $\nabla$ \eqref{torsionBianchi}, we obtain the Bianchi identity for $\nabla^\lambda$
\begin{equation*}
\cyclic{X,Y,Z}R(X,Y,Z,V)=\sigma_T(X,Y,Z,V)+\cyclic{X,Y,Z}\Delta^\lambda(V,T^\lambda(X,Y),Z)\, ,
\end{equation*}
where $\sigma_T(X,Y,Z,V)$ is defined in  \eqref{torsionBianchi}. For $X\in\mathcal{V}$, $Y,Z,V\in\mathcal{H}$ we find
\begin{align*}
2R(X,Y,Z,V)&=2R(X,Y,Z,V)-2R(Z,V,X,Y)=\cyclic{X,Y,Z,V}\left(\cyclic{X,Y,Z}R(X,Y,Z,V)\right)\\
&=\cyclic{X,Y,Z,V}\left(\cyclic{X,Y,Z}\Delta^\lambda(V,T^\lambda(X,Y),Z)\right)\, ,
\end{align*}
as the cyclic sum over all entries of a four-form vanishes. Then a case by case study of \eqref{Deltalambda} and \eqref{lambdatorsion} shows that if exactly one of $X,Y,Z,V$ is vertical while the others are horizontal, then $\Delta^\lambda(V,T^\lambda(X,Y),Z)=0$.
\end{proof}

We denote the curvature operator $\mathcal{R}^\lambda\in\Lambda^2T^*M\otimes\mathrm{End}(TM)$ given by
\begin{equation*}
g(\mathcal{R}^\lambda(X,Y)Z,V)=R^\lambda(X,Y,Z,V)\, .
\end{equation*}

\begin{lem}
Let $\mathcal{R}^\lambda_1\in (\Lambda^2\mathcal{V}\oplus\Lambda^2\mathcal{H})\otimes(\mathrm{End}(\mathcal{V})\oplus\mathrm{End}(\mathcal{H}))\subset \Lambda^2T^*M\otimes\mathrm{End}(TM)$ be given by
\begin{equation*}
\mathcal{R}^\lambda_1=-\begin{bmatrix} 4\alpha-\lambda &2\alpha\\2\alpha-\frac\lambda2 &\alpha\end{bmatrix}\sum_{i=1}^3\Phi_i\otimes\phi_i\, ,
\end{equation*} 
where the matrix is understood with respect to $\Lambda^2\mathcal{V}\oplus\Lambda^2\mathcal{H}\to\mathrm{End}(\mathcal{V})\oplus\mathrm{End}(\mathcal{H})$, that is: the entries of the matrix indicate the coefficients of the four different components of the operator $\Phi_i\otimes \phi_i$. Then, the curvature operator $\mathcal{R}^\lambda\in\Lambda^2T^*M\otimes\mathrm{End}(TM)$ is
\begin{equation*}
\mathcal{R}^\lambda=(\beta+\lambda)\mathcal{R}^\lambda_1+\mathcal{R}_2\, ,
\end{equation*}
where $\mathcal{R}_2\in \Lambda^2\mathcal{H}\otimes\mathrm{End}(\mathcal{H})$ is related to the Riemannian curvature operator $\mathcal{R}^{g_N}$ of the base by
\begin{equation*}
\mathcal{R}_2=\mathcal{R}^{g_N}+2\alpha\delta\sum_{i=1}^3\Phi^\mathcal{H}_i\otimes\phi_i|_\mathcal{H}\, .
\end{equation*}
In particular, $\mathcal{R}_2$ is independent of $\lambda$.
\end{lem}

\begin{proof}
We observe that by \eqref{nablacurv} all but the $\Lambda^2\mathcal{H}\to\mathrm{End}(\mathcal{H})$-component of $\mathcal{R}^\lambda$ are included in $(\beta+\lambda)\mathcal{R}^\lambda_1$. Then by definition $\mathcal{R}_2\coloneqq \mathcal{R}^\lambda-(\beta+\lambda)\mathcal{R}^\lambda_1\in\Lambda^2\mathcal{H}\otimes\mathrm{End}(\mathcal{H})$. From \eqref{projcurv} we have for $X,Y\in \mathcal{H}$
\begin{align*}
\mathcal{R}_2(X,Y)&=\mathcal{R}^\lambda(X,Y)|_\mathcal{H}+\alpha(\beta+\lambda)\sum_{i=1}^3\Phi_i(X,Y)\phi_i|_\mathcal{H}\\
&=\mathcal{R}^{g_N}(X,Y)+\alpha((4\alpha-\lambda)+(\beta+\lambda))\sum_{i=1}^3\Phi_i(X,Y)\phi_i|_\mathcal{H}\\
&=\mathcal{R}^{g_N}(X,Y)+2\alpha\delta\sum_{i=1}^3\Phi_i(X,Y)\phi_i|_\mathcal{H}\, .\qedhere
\end{align*}
\end{proof}

\subsection{Solving the heterotic G$_2$ system}

Our discussion so far has been completely general, but we are particularly interested in the case $n=1$ corresponding to $7$ dimensions. We therefore fix the dimension for the rest of the section.

The spinorial description of 3-$(\alpha,\delta)$-Sasaki manifolds in dimension $7$ is well understood \cite{Agricola:2018}. As shown in \cite{Agricola:2023spin}, in this case the bundle $E$ has rank $3$ and is spanned by the \emph{auxiliary} spinors, see \Cref{app:Spnstructuresandspinors} for a more explicit description. They are conventionally denoted as $\psi_i$, where one should note that $\psi_i\in E_j\cap E_k$, but $\psi_i\notin E_i$. They are $\mathcal{H}$-Killing spinors, and using \eqref{eq:formulasSp1auxiliaryspinors} in \eqref{eq:Hkillingspinor} we see they satisfy:
\begin{equation*}
	\nabla^g_X\psi_i=\frac{\alpha}{2} X\cdot\psi_i\ \; \text{for }X\in\mathcal{H}\, , \quad \nabla^g_{\xi_i}\psi_i=\frac{2\alpha-\delta}{2} \xi_i\cdot\psi_i\, , \quad \nabla^g_{\xi_j}\psi_i=\frac{3\delta-2\alpha}{2} \xi_j\cdot\psi_i\, , \; \text{for }j\neq i\, .
\end{equation*}
In addition, 3-$(\alpha,\delta)$-Sasaki manifolds possess a fourth generalized Killing spinor determined by $\psi_0=-\xi_i\cdot\psi_i$ for any of the auxiliary spinors. This is known as the \emph{canonical} spinor and from \eqref{eq:formulasSp1canonicalspinor} applied to \eqref{eq:Hkillingspinor} it satisfies
\begin{equation*}
	\nabla^g_X\psi_0=-\frac{3\alpha}{2} X\cdot\psi_0\ \quad \text{for }X\in\mathcal{H}\, , \qquad \nabla^g_Y\psi_0=\frac{2\alpha-\delta}{2} Y\cdot\psi_0 \quad \text{for }Y\in\mathcal{V}\, .
\end{equation*}
The canonical spinor is essential to us since it generates via \eqref{eq:associativeformintermsofspinors} the \emph{canonical} G$_2$-structure $\varphi$ of a 7-dimensional $3$-$(\alpha,\delta)$-Sasaki manifold. This is the G$_2$-structure given precisely by the formula \eqref{eq:G2structure3alphadelta}. Note that, with respect to the standard volume form $e^{1\cdots7}=-\frac{1}{3!}\eta_i\wedge\Phi_i^3$, the coassociative four-form is given by
\begin{equation}
\label{eq:coassociativeform3alphadelta}
    \psi=\cyclic{i,j,k}\Phi_i^\mathcal{H}\wedge\eta_{jk}+ \frac{1}{6}\sum_{\ell=1}^3\Phi_\ell^\mathcal{H}\wedge\Phi_\ell^\mathcal{H} \, .
\end{equation}
It was shown in \cite{Agricola:2018} that the canonical G$_2$-structure is coclosed and that its associated characteristic connection agrees with the canonical connection from \Cref{def:canonicalconnection3alphadelta}. As a result, the characteristic G$_2$-connection has reducible holonomy inside Sp(1)Sp(1).

\begin{prop}
    The canonical G$_2$-structure $\varphi$ in \eqref{eq:G2structure3alphadelta} is coclosed with torsion classes
    \begin{equation}
    \label{eq:3alphadeltatorsionclasses}
    \tau_1=\tau_2=0\, , \qquad \tau_0=\frac{12}{7}(2\alpha+\delta)\, , \qquad \tau_3=(10\alpha-2\delta)\left( \eta_{123}-\frac{1}{7}\varphi \right)\, .
\end{equation}
\end{prop}
\begin{proof}
    Using \eqref{def3ad} it is immediate to compute $\dd\varphi$ and $\dd\psi$ and identify the torsion classes. Alternatively, one can obtain the torsion classes directly from the spinors using the formulas from \cite{Agricola:2014}.
\end{proof}

We will later require the exterior derivative of the torsion \eqref{eq:torsiong2}:
\begin{lem}[\cite{Agricola:2018}]
    The torsion of the characteristic G$_2$-connection is
\begin{equation*}
    T^c= 2(\delta-4\alpha)\eta_{123}+ 2\alpha\sum_{\ell=1}^3\eta_\ell\wedge\Phi_\ell^\mathcal{H} \, ,
\end{equation*}
and its exterior derivative is
\begin{equation}
\label{eq:3alphadeltadtorsion}
\dd T^c= 4\alpha\beta\,\cyclic{i,j,k}\Phi_i^\mathcal{H}\wedge\eta_{jk}+4\alpha^2\sum_{\ell=1}^3\Phi_\ell^\mathcal{H}\wedge\Phi_\ell^\mathcal{H} \, .
\end{equation}
\end{lem}

\begin{rem}
    Each of the auxiliary spinors $\psi_i$ gives rise to a different G$_2$-structure $\varphi_i\coloneqq\eta_{123}+2\eta_i\wedge\Phi^\mathcal{H}_i-\sum_{\ell=1}^3\eta_\ell\wedge\Phi_\ell^\mathcal{H}$ which can also be considered. Each $\varphi_i$ becomes nearly parallel for the choice of parameters $\alpha=\delta$ corresponding to a $3$-$\alpha$-Sasaki manifold.

    The canonical G$_2$-structure, on the other hand, becomes nearly parallel if and only if $\delta=5\alpha$. In the case of squashed $3$-Sasaki manifolds, this choice of parameters corresponds to a metric that,  in the literature, is sometimes referred to as the ``other'' Einstein metric.
\end{rem}

In dimension $7$ the dimension of $\mathcal{H}$---or equivalently the dimension of the base $N$ of the canonical submersion---is $4$. Therefore, we may split the space of two-forms $\Lambda^2\mathcal{H}=\Lambda^2N=\Lambda^2_+\oplus\Lambda^2_-$ into self-dual and anti-self-dual two-forms. Note that in this case $\Lambda^2_+\wedge\Lambda^2_-=0$. Furthermore, $\{\Phi^\mathcal{H}_i\}$ is a basis of $\Lambda^2_+$ that satisfies $\Phi^\mathcal{H}_i\wedge\Phi^\mathcal{H}_j=0$ for all $i\neq j$.

\begin{lem}
We have $\Lambda^2\mathcal{V}\oplus\Lambda^2_+\subset\ker \mathcal{R}_2$ and $\Lambda^2_-\subset\ker\mathcal{R}_1^\lambda$.
\end{lem}

\begin{proof}
Since $\mathcal{R}_2\in \Lambda^2\mathcal{H}\otimes\mathrm{End}(\mathcal{H})$ we have $\Lambda^2\mathcal{V}\in\ker(\mathcal{R}_2)$. Note that for any orthonormal coframe $\lbrace e^\mu\rbrace_{\mu=1}^{4n+3}$ the horizontal part of the fundamental form is given as $\Phi_i^\mathcal{H}=-\frac 14\sum_{r=4}^{4n+3}\left(e^r\wedge\phi_ie^r+\phi_je^r\wedge\phi_ke^r\right)$. Additionally, we observed in \Cref{projcurvlemma} that $R^\lambda|_{\Lambda^2\mathcal{H}\otimes\Lambda^2\mathcal{H}}$ has pair symmetry. Applying \eqref{RonPhi} and $|\Phi_i^\mathcal{H}|^2=2n$ we thus find
\begin{align*}
\mathcal{R}^\lambda(\Phi_i^\mathcal{H})\vert_\mathcal{H}&=-\frac14\sum_{r=4}^{4n+3}\left(\mathcal{R}^\lambda(e^r\wedge\phi_i e^r)\vert_\mathcal{H}+\mathcal{R}^\lambda(\phi_j e^r\wedge\phi_k e^r)\vert_\mathcal{H}\right)\\
&=\frac 12\alpha(\beta+\lambda)4n\,\phi_i\vert_\mathcal{H}
=(\beta+\lambda)\mathcal{R}_1^\lambda(\Phi_i^\mathcal{H})\vert_\mathcal{H}\, ,
\end{align*}
and we find $\Phi_i^\mathcal{H}\in \ker \mathcal{R}_2$. On the other hand, we have $\Phi_i(\Lambda^2_-)=0$ proving the final claim.
\end{proof}

\begin{prop}
\label{prop:g2instanton3alphadelta}
Let $(M,\xi_i,\eta_i,\phi_i,g)_{i=1,2,3}$ be a $3$-$(\alpha,\delta)$-Sasaki manifold. The family of connections $\nabla^{\lambda}$ satisfies
\begin{equation}
\label{eq:obstructionto3alphadeltaSasakiG2instanton}
    \mathcal{R}^\lambda\wedge\psi= -\left(\beta+\lambda\right)\frac{\lambda}{2} \cyclic{i,j,k}\left(\Phi_i^\mathcal{H}\wedge\Phi_i^\mathcal{H}\wedge\eta_{jk}\otimes\left(\phi_i|_\mathcal{V}+\frac 12\phi_i|_\mathcal{H}\right)\right) \, .
\end{equation}
Thus, $\nabla^{\lambda}$ is a G$_2$-instanton for the canonical G$_2$-structure if and only if $\lambda\in\{0,-\beta\}$.
\end{prop}
\begin{proof}
Note from \eqref{eq:coassociativeform3alphadelta} that $\psi\in\Lambda^2_+\wedge\Lambda^2_+\oplus\Lambda^2_+\wedge\Lambda^2\mathcal{V}$, so we immediately have $\mathcal{R}_2\wedge\psi=0$. Observing that $\Phi_i^\mathcal{H}\wedge\Phi_i^\mathcal{H}=2\mathrm{dvol}_\mathcal{H}$ for all $i\in\lbrace 1,2,3 \rbrace$, we find
\begin{align*}
\mathcal{R}^\lambda_1\wedge\psi&=\cyclic{i,j,k}\left((4\alpha-\lambda)\eta_{jk}\wedge\frac 12\Phi_i^\mathcal{H}\wedge\Phi_i^\mathcal{H}-2\alpha\Phi_i^\mathcal{H}\wedge\Phi_i^\mathcal{H}\wedge\eta_{jk}\right)\otimes\phi_i|_\mathcal{V}\\
&\qquad +\cyclic{i,j,k}\left(\left(2\alpha-\frac\lambda 2\right)\eta_{jk}\wedge\frac 12\Phi_i^\mathcal{H}\wedge\Phi_i^\mathcal{H}-\alpha\Phi_i^\mathcal{H}\wedge\Phi_i^\mathcal{H}\wedge\eta_{jk}\right)\otimes\phi_i|_\mathcal{H}\\
&=-\,\frac{\lambda}{2}\cyclic{i,j,k}\Phi_i^\mathcal{H}\wedge\Phi_i^\mathcal{H}\wedge\eta_{jk}\otimes\left(\phi_i|_\mathcal{V}+\frac 12\phi_i|_\mathcal{H}\right)\, .
\end{align*}
The result follows as $\mathcal{R}^\lambda=(\beta+\lambda)\mathcal{R}^\lambda_1+\mathcal{R}_2\, $.
\end{proof}

\begin{rem}
The exact G$_2$-instantons coincide if and only if $\beta=0$, or equivalently $\delta=2\alpha$. For $\lambda=0$ the instanton is precisely the canonical connection $\nabla$ of the given $3$-$(\alpha,\delta)$-Sasaki manifold. For positive $3$-$(\alpha,\delta)$-Sasaki manifolds the $\lambda=-\beta$ G$_2$-instanton can be understood as the canonical connection of the $\mathcal{H}$-homothetic deformation into a parallel $3$-$(\alpha,\delta)$-Sasaki manifold.    
\end{rem}

\begin{lem}\label{trace}
\begin{equation}\label{eq:trace}
\operatorname{tr}(\mathcal{R}^\lambda\wedge\mathcal{R}^\lambda)=12\alpha(\beta+\lambda)^2\cyclic{i,j,k}\left((4\alpha-\lambda)\eta_{jk}\wedge\Phi_i^\mathcal{H}-\alpha\Phi_i^\mathcal{H}\wedge\Phi_i^\mathcal{H}\right) +\operatorname{tr}(\mathcal{R}_2\wedge\mathcal{R}_2)\, .
\end{equation}
\end{lem}

\begin{proof}
We identify $\mathrm{End}_{\mathrm{skew}}(TM)$ with $\Lambda^2T^*M$ via the metric. Then, the image of $\mathcal{R}^\lambda_1$ lies in $\Lambda^2_+\oplus\Lambda^2\mathcal{V}$ whereas the image of $\mathcal{R}_2$ lies in $\Lambda^2_-$. Since these spaces are orthogonal with respect to the metric in $\Lambda^2$, so are their corresponding endomorphisms with respect to the trace metric. In particular, $\operatorname{tr}(\mathcal{R}_1^\lambda\wedge\mathcal{R}_2)=0\, $. Thus,
\begin{equation*}
\operatorname{tr}(\mathcal{R}^\lambda\wedge\mathcal{R}^\lambda)=(\beta+\lambda)^2\operatorname{tr}(\mathcal{R}_1^\lambda\wedge\mathcal{R}_1^\lambda)+\operatorname{tr}(\mathcal{R}_2\wedge\mathcal{R}_2)\, .
\end{equation*}
We compute
\begin{align*}
\operatorname{tr}(\mathcal{R}_1^\lambda\wedge\mathcal{R}_1^\lambda|_\mathcal{H})
&=\cyclic{i,j,k}\!\!\left(\!\alpha^2\Phi_i^\mathcal{H}\wedge\Phi_i^\mathcal{H}-\alpha\,\frac12(4\alpha-\lambda)\Phi_i^\mathcal{H}\wedge\eta_{jk}-\frac 12(4\alpha-\lambda)\alpha\,\eta_{jk}\wedge\Phi_i^\mathcal{H}\right)\!\operatorname{tr}(\phi_i^2|_\mathcal{H})\\
&=-4\alpha\cyclic{i,j,k}\left(\alpha\Phi_i^\mathcal{H}\wedge\Phi_i^\mathcal{H}-(4\alpha-\lambda)\Phi_i^\mathcal{H}\wedge\eta_{jk}\right)\, ,  \\
\operatorname{tr}(\mathcal{R}_1^\lambda\wedge\mathcal{R}_1^\lambda|_\mathcal{V})
&=\cyclic{i,j,k}\left(4\alpha^2\Phi_i^\mathcal{H}\wedge\Phi_i^\mathcal{H}-2\alpha(4\alpha-\lambda)\Phi_i^\mathcal{H}\wedge\eta_{jk}-(4\alpha-\lambda)2\alpha\,\eta_{jk}\wedge\Phi_i^\mathcal{H}\right)\operatorname{tr}(\phi_i^2|_\mathcal{V})\\
&=-8\alpha\cyclic{i,j,k}\left(\alpha\Phi_i^\mathcal{H}\wedge\Phi_i^\mathcal{H}-(4\alpha-\lambda)\Phi_i^\mathcal{H}\wedge\eta_{jk}\right)\, ,
\end{align*}
and adding them we obtain \eqref{eq:trace}.
\end{proof}

\begin{thm}
Let $\alpha'>0$ and let $(M,g,\xi_i,\eta_i,\phi_i)_{i=1,2,3}$ be a degenerate $7$-dimensional $3$-$(\alpha,\delta)$-Sasaki manifold with its canonical G$_2$-structure $\varphi$ and canonical connection $\nabla$ with torsion $T$. If $\alpha^2=\frac{1}{12\alpha'}$, then 
\begin{equation*}
[(M,\varphi),(TM,\nabla^{-\beta}),(TM,\nabla),T]\, ,
\end{equation*}
where $\beta=2(\delta-2\alpha)$, is a solution to the heterotic G$_2$ system.
\end{thm}

\begin{proof}
We have already seen that $\nabla=\nabla^0$ and $\nabla^{-\beta}$ are G$_2$-instantons. Hence, we only have to check the heterotic Bianchi identity. On the right hand side we have by \eqref{eq:trace}
\begin{align*}
\frac{\alpha'}4(\operatorname{tr}&(\mathcal{R}^{-\beta}\wedge\mathcal{R}^{-\beta})-\operatorname{tr}(\mathcal{R}^0\wedge\mathcal{R}^0))\\
&=\frac{\alpha'}4(\operatorname{tr}(\mathcal{R}_2\wedge\mathcal{R}_2)-12\alpha\beta^2\cyclic{i,j,k}(4\alpha\,\eta_{jk}\wedge\Phi_i^\mathcal{H}-\alpha\Phi_i^\mathcal{H}\wedge\Phi_i^\mathcal{H}) -\operatorname{tr}(\mathcal{R}_2\wedge\mathcal{R}_2))\\
&=3\alpha^2\beta^2\alpha'\cyclic{i,j,k}\left(\Phi_i^\mathcal{H}\wedge\Phi_i^\mathcal{H}-4\eta_{jk}\wedge\Phi_i^\mathcal{H}\right) \, ,
\end{align*}
whereas the left-hand side is given by \eqref{eq:3alphadeltadtorsion}. Equating the coefficients of the forms we get two equations
\begin{equation*}
4\alpha^2=3\alpha^2\beta^2\alpha'\,, \qquad\qquad
4\alpha\beta=-12\alpha^2\beta^2\alpha'\, ,
\end{equation*}
which have a real solution if and only if $\delta=0$. In this case, $\alpha^2=\frac1{12\alpha'}$.
\end{proof}

\begin{rems}
    We obtain exact solutions for arbitrary parameter $\alpha'$. Using \eqref{eq:3alphadeltatorsionclasses} we see that in our solutions $\tau_0^2\sim\frac{1}{\alpha'}$, and by \Cref{rem:cosmologicalconstant} this means that the 3-dimensional cosmological constant is inversely proportional to the string parameter, $\Lambda\sim\frac{1}{\alpha'}$. This means that in the $\alpha'\to 0$ limit, the AdS$_3$ radius goes to zero and the spacetime becomes singular. This feature is also present in other solutions to the heterotic system in the literature, see for example \cite{delaOssa:2021qlt}.

     It is interesting to study the behaviour of the string parameter under the $\mathcal{H}$-homothetic deformations introduced in \Cref{prop:Hhomotheticdeformations}. With respect to global scaling, i.e. $a=c^2$, taking $c\to 0$ we have $\alpha\to \infty$, implying $\alpha'\to 0$. Thus, a collapsing limit for $M$ corresponds to a small string parameter. In addition, any $\mathcal{H}$-homothetic deformation with $\frac ca \to \infty$ recovers the limit $\alpha'\to 0$. This means that a small string parameter can also be obtained by sufficiently enlarging the fibres of the $3$-$(\alpha,\delta)$-Sasaki manifold relative to the base space.
\end{rems}

Apart from exact solutions we can consider approximate solutions in the sense of \Cref{def:approximate}. We will be able to find these on both positive and negative $3$-$(\alpha,\delta)$-Sasaki manifolds. We first study which pairs of connections $(A,\Theta)$---with $A$ being one of the instantons in \Cref{prop:g2instanton3alphadelta}---provide solutions to the heterotic Bianchi identity:

\begin{prop}
\label{prop:solutionsbianchi3alphadelta}
Let $\alpha'>0$ and let $(M,g,\xi_i,\eta_i,\phi_i)_{i=1,2,3}$ be a  $7$-dimensional $3$-$(\alpha,\delta)$-Sasaki manifold. Then the following pairs of connections solve the heterotic Bianchi identity \eqref{eq:heteroticBianchiidentity} with $H=T$:
\begin{enumerate}[i)]
\item $(A,\Theta)=\left(\nabla^{-\beta},\nabla^\lambda\right)$ with $\lambda=2\delta$ and the additional condition $\frac 1{\alpha'}=12(\delta-\alpha)^2\,$.
\item $(A,\Theta)=\left(\nabla,\nabla^\lambda\right)$ with $\lambda=-\beta\pm\sqrt{\beta^2+\frac{4}{3\alpha'}}$, provided that $\alpha$ and $\delta$ can be chosen such that $\frac{4}{3\alpha'}=-\beta^2+2\delta(\delta+\beta)\pm 2\sqrt{\delta^3(\delta+2\beta)}\,$ for some choice of sign.
\end{enumerate}
\end{prop}

\begin{proof}
In both cases the left-hand side of the Bianchi identity is given by \eqref{eq:3alphadeltadtorsion}, and using \Cref{trace} we see that the $\mathcal{R}_2$ terms on the right-hand side cancel out. Equating the coefficients of the forms, we obtain a system of two algebraic equations.

For i), the equations are
\begin{equation*}
    4\alpha^2=3\alpha'\alpha^2(\beta+\lambda)^2\, , \qquad \qquad 4\alpha\beta=3\alpha'\alpha(\beta+\lambda)^2(\lambda-4\alpha) \, ,
\end{equation*}
and these are easily solved to yield the result.

For ii), the equations are
\begin{equation*}
    4\alpha^2=3\alpha'\alpha^2((\beta+\lambda)^2-\beta^2)\, , \qquad 4\alpha\beta=-3\alpha'\alpha(\lambda(\beta+\lambda)^2-4\alpha((\beta+\lambda)^2-\beta^2)) \, .
\end{equation*}
Solving the first equation provides the desired value of $\lambda$. Plugging this into the second equation gives, after some manipulation, a quadratic equation for $\frac{4}{3\alpha'}$ whose solution is the constraint we wanted.
\end{proof}

\begin{rem}
    The heterotic Bianchi identity can never be solved with our families of connections for the choice $\alpha=\delta$ corresponding to a $3$-$\alpha$-Sasaki manifold. Therefore, including deformations beyond global rescaling proves essential to obtain solutions.
\end{rem}

We obtain approximate solutions for the first case in \Cref{prop:solutionsbianchi3alphadelta}.

\begin{thm}
    Let $\alpha'>0$ and let $(M,g,\xi_i,\eta_i,\phi_i)_{i=1,2,3}$ be a $7$-dimensional $3$-$(\alpha,\delta)$-Sasaki manifold with its canonical G$_2$-structure $\varphi$ and canonical connection $\nabla$ with torsion $T$. Consider $\delta=\delta(\alpha')$ such that $\delta=\mathcal{O}(\alpha')^{\frac{5}{2}}$ as $\alpha'\to 0$, and take $\alpha=\delta-\frac{1}{\sqrt{12\alpha'}}$. Then, for the choice $\lambda=2\delta$ the quadruple
\begin{equation*}
[(M,\varphi),(TM,\nabla^{-\beta}),(TM,\nabla^\lambda),T]
\end{equation*}
is an approximate solution to the heterotic G$_2$ system.
\end{thm}
\begin{proof}
    By \Cref{prop:solutionsbianchi3alphadelta}, the chosen connections satisfy the heterotic Bianchi identity. On the other hand, \Cref{prop:g2instanton3alphadelta} shows that $\nabla^{-\beta}$ is a G$_2$-instanton and it only remains to verify that the instanton condition on the connection $\nabla^\lambda$ is satisfied up to first order in $\alpha'$. From \eqref{eq:obstructionto3alphadeltaSasakiG2instanton} we obtain
    \begin{equation*}
        \vert\mathcal{R}^{\lambda}\wedge\psi\vert_g= 8\,\vert\left(\delta-\alpha\right)\delta\vert  \, \left\vert \mathrm{dvol}_\mathcal{H}\wedge \sum_{i=1}^3\left(\eta_{jk}\otimes(\Phi_i|_\mathcal{V}+\frac 12\Phi_i|_\mathcal{H})\right) \right\vert_g = 48\,\vert\delta-\alpha\vert\,\vert\delta\vert \, ,
    \end{equation*}
    As $\alpha'\to 0\,$, we have that $\vert\delta-\alpha\vert=\frac{1}{\sqrt{12\alpha'}}=\mathcal{O}(\alpha')^{-\frac{1}{2}}$ and $\vert\delta\vert=\mathcal{O}(\alpha')^{\frac{5}{2}}$. This immediately ensures that \eqref{eq:approximateG2instantoncondition} is satisfied and finishes the proof.
\end{proof}

\begin{rems}
    There are several possible choices of $\delta$ that yield valid approximate solutions as $\alpha'\to 0$, the simplest one being $\delta=\left(\alpha'\right)^{\frac{5}{2}}$.

    Note from \eqref{eq:3alphadeltatorsionclasses} that for these solutions the torsion class $\tau_0$ blows up as  $\alpha'\to 0$. This is the same behaviour as the exact solution we obtained earlier.

    Finding an approximate G$_2$-instanton for the second case in \Cref{prop:solutionsbianchi3alphadelta} does not seem possible. The reason is that the term $\sqrt{\beta^2+\frac{4}{3\alpha'}}$ is always at least $\mathcal{O}(\alpha')^{-\frac{1}{2}}$ as $\alpha'\to 0$, so one would need $\lambda=\mathcal{O}(\alpha')^{\frac{5}{2}}$ as $\alpha'\to 0$. However, this appears to be unfeasible since we must also impose the intricate relation between $\alpha'$, $\delta$ and $\beta$ we found earlier.
\end{rems}

\section{Characteristic holonomy SU(3)}
\label{sec:su3case}

In this section we study the heterotic G$_2$ system on spin $\eta$-Einstein $\alpha$-Sasaki manifolds. To do so, we first describe these manifolds in terms of two constants $(\alpha,\delta)$ in complete analogy with the 3-$(\alpha,\delta)$-Sasaki manifolds in the previous section. We then introduce once again a 1-parameter family of connections $\nabla^\lambda$ on the tangent bundle, and after specializing to dimension $7$ we find approximate solutions to the heterotic G$_2$ system.

\subsection{$(\alpha,\delta)$-Sasaki manifolds}

We begin with the following definition:

\begin{defi}\label{def:adSas}
    A $2m+1$-dimensional manifold $M$ is \emph{$(\alpha,\delta)$-Sasaki} if it has an $\mathrm{SU}(m)$-structure $(\xi,\eta,\phi,g,\Phi,\Omega)$ (see \Cref{def:SUmstructure}) satisfying
    \begin{equation}
    \label{eq:exteriorderivativesSUm}
        \dd\eta=2\alpha\,\Phi\, , \qquad \dd\Omega=(m+1)i\delta \, \eta\wedge\Omega \, ,
    \end{equation}
    where $\alpha$ and $\delta$ are real constants with $\alpha\neq 0$, and the structure corresponds to a generalized Killing spinor in the sense of \Cref{app:SUmstructuresandspinors}.
\end{defi}\begin{rem}
    In \cite{Conti:2007} the authors argue that the defining spinor of an $\mathrm{SU}(m)$-structure with $\dd \Phi=0$ and $\dd(\eta\wedge\Omega)=0$ should be generalized Killing. However, the statement is proved only in the real analytic case and in dimension $5$, see \cite{Conti:2005}. We expect the assumption on the spinor in the previous definition to be redundant.
\end{rem}

\begin{rem} 
    It is essential to note that $(\alpha,\delta)$-Sasaki manifolds are precisely the spin $\eta$-Einstein $\alpha$-Sasaki manifolds, so \Cref{def:adSas} is equivalent to \Cref{def:etaEinstein} in this case. Indeed: as observed in \cite{Conti:2008} (see also \cite{Conti:2005} for the 5-dimensional case), the $\mathrm{U}(m)$-structure associated to the $\alpha$-Sasaki structure of a spin $\eta$-Einstein $\alpha$-Sasaki manifold can be further reduced to an $\mathrm{SU}(m)$-structure. In addition, \cite{Conti:2008} shows that the intrinsic torsion of the $\mathrm{SU}(m)$-structure takes a very particular form and depends on two real parameters $(\lambda,\mu)$. Conversely, we will show that an $(\alpha,\delta)$-Sasaki manifold is $\eta$-Einstein $\alpha$-Sasaki. One then can check that $(\lambda,\mu)=(2\alpha,-(m+1)\delta)$. We will refer to spin $\eta$-Einstein $\alpha$-Sasaki manifolds as $(\alpha,\delta)$-Sasaki manifolds to stress the dependence on the parameters $(\alpha,\delta)$.
\end{rem}

\begin{rem}
    Decomposing the holomorphic volume form in its real and imaginary parts, $\Omega=\Omega_++i\,\Omega_-\,$, the last equation in \eqref{eq:exteriorderivativesSUm} reads
    \begin{equation*}
        \dd\Omega_+=-(m+1)\delta \, \eta\wedge\Omega_- \, , \qquad \dd\Omega_-=(m+1)\delta \, \eta\wedge\Omega_+ \, ,
    \end{equation*}
\end{rem}
By definition, an $(\alpha,\delta)$-Sasaki manifold is determined by a generalized Killing spinor $\Psi$. The following proposition expands on the specific form of that Killing spinor.
\begin{prop}
    Let $M$ be a $2m+1$-dimensional $(\alpha,\delta)$-Sasaki manifold. Then, $M$ admits a generalized Killing spinor $\Psi$ satisfying
    \begin{equation}
    \label{eq:genkillspinorSUm}
        \nabla^g_X\Psi=(-1)^{m+1}\,\frac{\alpha}{2}\, X\cdot\Psi\,, \qquad \nabla^g_\xi\Psi=(-1)^m\frac{ m\,\alpha-(m+1)\delta}{2}\,\xi\cdot\Psi\,,
    \end{equation}
    where $X\in\mathcal{H}=\ker(\eta)$ and $\xi$ is the Reeb vector field.

    Conversely, if $(M,g)$ is a $2m+1$-dimensional (oriented) Riemannian spin manifold with a nowhere-vanishing pure Dirac spinor $\Psi$ satisfying \eqref{eq:genkillspinorSUm}---where $\alpha$ and $\delta$ are real constants with $\alpha\neq 0$---then $M$ is an $(\alpha,\delta)$-Sasaki manifold.
\end{prop}
\begin{proof}
    We first prove the converse: assume that $M$ has a nowhere-vanishing pure Dirac spinor $\Psi$ satisfying \eqref{eq:genkillspinorSUm}, which means
    \begin{equation*}
        S(\xi)=(-1)^m\left( m\alpha-(m+1)\delta \right)\xi\, , \qquad S(X)=(-1)^{m+1}\alpha\, X \, \qquad \text{for all } X\in\mathcal{H} \, ,
    \end{equation*}
    where $S$ is the endomorphism introduced in \Cref{def:generalizedKillingspinor}. As explained in \Cref{app:SUmstructuresandspinors}, the spinor $\Psi$ defines an $\mathrm{SU}(m)$-structure on $M$ whose intrinsic torsion is encoded by $\nabla^g\Psi$ \cite{Conti:2007}. We can now use the formulas in \Cref{lem:lemma1ContiFinno} to compute the exterior derivatives of $\eta$ and $\Omega$. For example, working in the adapted frame given in \eqref{eq:explicitSU3forms} we have
    \begin{equation*}
        \dd\eta=\sum_{\mu=1}^{2m+1}e^\mu\wedge\nabla^g_{e_\mu}\eta=\sum_{\mu=1}^{2m+1}e^\mu\wedge\left( (-1)^{m+1}  S(e_\mu)\lrcorner\Phi \right)=\sum_{a=2}^{2m+1} e^a\wedge \left( \alpha \, e_a\lrcorner\Phi \right)=2\alpha\,\Phi \, . 
    \end{equation*}
    The computation of $\dd\Omega$ is analogous and we recover the expressions \eqref{eq:exteriorderivativesSUm}, showing that $M$ is an $(\alpha,\delta)$-Sasaki manifold.

    Assume now we have an $(\alpha,\delta)$-Sasaki manifold $M$. Its $\mathrm{SU}(m)$-structure is equivalently described by a generalized Killing spinor $\Psi$ whose derivative encodes the intrinsic torsion of the $\mathrm{SU}(m)$-structure. Since the forms $\Phi$ and $\Omega$ satisfy \eqref{eq:exteriorderivativesSUm}, the computation above shows that $\nabla^g\Psi$ must be of the form \eqref{eq:genkillspinorSUm}, finishing the proof.
\end{proof}
\begin{cor}
    An $(\alpha,\delta)$-Sasaki manifold is $\alpha$-Sasaki, in particular the Nijenhuis-tensor $N$ vanishes.
\end{cor}
\begin{proof}
    Using the specific form of the generalized Killing spinor and \Cref{lem:lemma1ContiFinno} we have $\nabla^g_X\Phi=\alpha\,\eta\wedge X^\flat$. By the Chinea-Gonzalez classification \cite{Chinea:1990}, we have that $(M,g,\xi,\eta,\phi)$ is $\alpha$-Sasaki.
\end{proof}

\begin{rem}
    Every $(\alpha,\delta)$-Sasaki manifold admits a second generalized Killing spinor $\bar{\Psi}$ satisfying
    \begin{equation}
    \label{eq:othergenkillspinorSUm}
        \nabla^g_X\bar{\Psi}=\frac{\alpha}{2}\, X\cdot\bar{\Psi}\,, \qquad \nabla^g_\xi\bar{\Psi}=-\,\frac{ m\,\alpha-(m+1)\delta}{2}\, \xi\cdot\bar{\Psi}\,,
    \end{equation}
    where $X\in\mathcal{H}=\ker(\eta)$ and $\xi$ is the Reeb vector field.

      Both $\Psi$ and $\bar{\Psi}$ are \emph{Sasakian quasi-Killing} spinors in the sense of \cite{Kim:1999} with $(a,b)=((-1)^{m+1}\frac{\alpha}{2},(-1)^m\frac{m+1}{2}\left( \alpha-\delta \right))$ for $\Psi$ and $(a,b)=(\frac{\alpha}{2},-\,\frac{m+1}{2}\left( \alpha-\delta \right))$ for $\bar{\Psi}$. Some of the results in \cite{Kim:1999} follow as particular cases of our description. For example, \cite[Theorem 6.14]{Kim:1999} corresponds to equations \eqref{eq:genkillspinorSUm} and \eqref{eq:othergenkillspinorSUm} for the choices $\alpha=1/a$ and $\delta=1$ (note however that with our conventions the endomorphism $S$ has an additional global minus sign compared to that of \cite{Kim:1999}).

      In particular, the case $\alpha=\delta=1$ corresponds to Sasaki--Einstein manifolds \cite{Sparks:2010sn}, which admit two Killing spinors with Killing constant of the same sign for odd $m$ and of opposite sign for even $m$.
\end{rem}
Employing the techniques of \cite{Kim:1999}, we can use the spinor $\Psi$ to compute the Ricci tensor in terms of the parameters $(\alpha,\delta)$. Our language simplifies the proof given in \cite{Kim:1999}.
\begin{prop}[\cite{Kim:1999}]
    The Ricci curvature of an $(\alpha,\delta)$-Sasaki manifold is given by
\begin{equation}
\label{eq:riccitensorsu3}
    \operatorname{Ric}^g(\alpha,\delta)=2\alpha((m+1)\delta-\alpha) g_{\mathcal{H}}+2m\alpha^2 g_{\mathcal{V}}\, .
\end{equation}
\end{prop}

\begin{proof}
    We use the 1/2-Ricci formula
    obtained in \cite{Kim:1999}:
    \begin{equation*}
	\frac{1}{2}\operatorname{Ric}^g(X)\cdot\Psi=D\left( \nabla^g_X\Psi \right)-\nabla^g_X\left( D\Psi \right)-\sum_{\mu=1}^{2m+1} e_\mu\cdot \nabla^g_{\nabla^g_{e_\mu} X}\Psi\, ,
\end{equation*}
    where here $X\in\Gamma(TM)$. The proof follows the idea of \cite[Lemma 6.4]{Kim:1999} and is most conveniently performed in an adapted frame. We first consider the case $X=\xi\,$. A direct computation using \eqref{eq:genkillspinorSUm} and the Clifford relations \eqref{eq:cliffordrelations} gives the action of the Dirac operator on $\Psi$
    \begin{align*}
        D\Psi & =\xi\cdot\nabla^g_\xi \Psi+\!\!\sum_{a=2}^{2m+1}\! e_a\cdot \nabla^g_{e_a}\Psi =(-1)^m\frac{ m\,\alpha-(m+1)\delta}{2}\, \xi\cdot\xi\cdot\Psi-(-1)^m\,\frac{\alpha}{2}\sum_{a=2}^{2m+1} e_a\cdot e_a\cdot\Psi \\
        & =-(-1)^m\frac{ m\,\alpha+(m+1)\delta}{2}\,\Psi\, ,
    \end{align*}
    and using \eqref{eq:genkillspinorSUm} again we find
    \begin{equation*}
        \nabla^g_\xi\left( D\Psi \right)=\left(\frac{m\,\alpha-(m+1)\delta}{2}\right)\left(\frac{m\,\alpha+(m+1)\delta}{2}\right) \xi\cdot\Psi\,,
    \end{equation*}
    To compute the other terms in the formula, we need the covariant derivative of $\xi$, which can be obtained from the formulas in \Cref{lem:extensionofContiFinnotoframe} and it is simply the usual covariant derivative of the Reeb vector field of an $\alpha$-Sasaki manifold
    \begin{equation*}
        \nabla^g_X\xi = -\alpha \, \phi(X) \, , 
    \end{equation*}
    where $X\in\Gamma(TM)$. The last term becomes
    \begin{equation*}
        \sum_{\mu=1}^{2m+1} e_\mu\cdot \nabla^g_{\nabla^g_{e_\mu} \xi}\Psi=-\alpha\sum_{a=2}^{2m+1} e_a\cdot \nabla^g_{\phi(e_a)} \Psi=(-1)^m\frac{\alpha^2}{2}\sum_{a=2}^{2m+1}  e_a\cdot \phi(e_a)\cdot \Psi=(-1)^{m+1}\alpha^2\, \Phi\cdot \Psi \, ,
    \end{equation*}
    and observing that
    \begin{align*}
        D\left(  \xi\cdot\Psi \right)&=\sum_{\mu=1}^{2m+1} e_\mu\cdot \nabla^g_{e_\mu} \left(  \xi\cdot\Psi \right)= \sum_{a=2}^{2m+1} e_a\cdot \left( -\alpha\,\phi(e_a) \right) \cdot\Psi + \sum_{\mu=1}^{2m+1} e_\mu\cdot  \xi\cdot\nabla^g_{e_\mu} \left( \Psi \right) \\
        &= 2\alpha\, \Phi \cdot\Psi + (-1)^m\frac{ m\,\alpha-(m+1)\delta}{2}\, \xi\cdot  \xi\cdot\xi \cdot \Psi + (-1)^{m+1} \, \frac{\alpha}{2} \sum_{a=2}^{2m+1} e_a\cdot \xi \cdot e_a \cdot\Psi \\
        &= 2\alpha\, \Phi \cdot\Psi + (-1)^{m+1}\frac{ m\,\alpha-(m+1)\delta}{2}\, \xi \cdot \Psi + (-1)^{m+1} \, \alpha \, m \, \xi \cdot\Psi \\
        &= 2\alpha\, \Phi \cdot\Psi + (-1)^{m+1} \frac{3m\,\alpha-(m+1)\delta}{2}\, \xi \cdot \Psi \, ,
    \end{align*}
    we obtain
    \begin{align*}
        D&\left(\nabla^g_\xi \Psi \right)=(-1)^m\frac{ m\,\alpha-(m+1)\delta}{2} \, D \left(  \xi\cdot\Psi \right) \\
        &=(-1)^m\left( m\,\alpha-(m+1)\,\delta \right)\alpha \, \Phi\cdot\Psi - \left(\frac{ m\,\alpha-(m+1)\delta}{2}\right) \left( \frac{ 3m\,\alpha-(m+1)\delta}{2} \right) \xi \cdot \Psi \, .
    \end{align*}
    Combining these formulas we find
    \begin{equation*}
	   \frac{1}{2}\operatorname{Ric}^g(\xi)\cdot\Psi=(-1)^m(m+1)\left(\alpha-\delta \right)\alpha \, \Phi\cdot\Psi -\left( m\,\alpha-(m+1)\delta \right) m\, \alpha \, \xi \cdot \Psi   \, .
    \end{equation*}
    Recall that $\Psi$ is a section of $\Sigma_0$ in the language of \cite{Kim:1999}, see \Cref{lem:lemma6.2Kim} for the characterization of $\Sigma_0$. This means that
     \begin{equation*}
         \Phi\cdot\Psi=m(-1)^m \, \xi \cdot \Psi \, ,
     \end{equation*}
     and we can therefore rewrite
     \begin{equation*}
	   \frac{1}{2}\operatorname{Ric}^g(\xi)\cdot\Psi=  m\, \alpha^2 \, \xi \cdot \Psi   \, .
    \end{equation*}
    We can then read the vertical component of the Ricci tensor. The computation for $X\in\mathcal{H}$ is analogous but it requires some knowledge of the covariant derivatives of the other terms in the frame. From \Cref{lem:extensionofContiFinnotoframe} we can write
    \begin{equation*}
       \nabla^g_X e_a =   -\alpha \, g(X,\phi(e_a)) \, \xi +\sum_{b=2}^{2m+1} g(\nabla^g_X e_a,e_b)\,e_b  \, ,  
    \end{equation*}
    for $X\in\Gamma(TM)$ and $a\in\lbrace 2,\dots,2m+1\rbrace$. The terms in the 1/2-Ricci formula are
    \begin{align*}
        D\left( \nabla^g_{e_a}\Psi \right) & =  (-1)^m \, \frac{\alpha^2}{2} \,  \phi(e_a)\cdot\xi\cdot\Psi + (-1)^{m+1} \,\frac{\alpha}{2} \, \sum_{\mu=1}^{2m+1}\sum_{b=2}^{2m+1} g(\nabla^g_X e_a,e_b)\,e_\mu\cdot e_b\cdot\Psi   \\ 
        & \ \ + \alpha \, \frac{ (m-2)\alpha+(m+1)\delta}{4} \,  e_a\cdot\Psi \, , \\
        \nabla^g_{a}\left( D\Psi \right) & = -\alpha\, \frac{ m\,\alpha+(m+1)\delta}{4}\, e_a\cdot\Psi\,, \\
        \sum_{\mu=1}^{2m+1} e_\mu\cdot \nabla^g_{\nabla^g_{e_\mu} e_{a}}\Psi & = (-1)^{m+1} \,\alpha \, \frac{ m\,\alpha-(m+1)\delta}{2} \, \phi(e_a)\cdot\xi\cdot\Psi  \\
        & \ \ + (-1)^{m+1} \,\frac{\alpha}{2} \, \sum_{\mu=1}^{2m+1}\sum_{b=2}^{2m+1} g(\nabla^g_X e_a,e_b)\,e_\mu\cdot e_b\cdot\Psi \, ,
    \end{align*}
    and note that the terms including $g(\nabla^g_X e_a,e_b)$ cancel out once these expressions are introduced in the 1/2-Ricci formula. We obtain
     \begin{equation*}
	   \frac{1}{2}\operatorname{Ric}^g(e_a)\cdot\Psi= (-1)^m \, \frac{m+1}{2} \, \alpha \, (\alpha-\delta) \, \phi(e_a)\cdot\xi\cdot\Psi  +\alpha \, \frac{ (m-1)\alpha+(m+1)\delta}{2} \,  e_a\cdot\Psi \, .
    \end{equation*}    
    The characterization of $\Sigma_0$ in \Cref{lem:lemma6.2Kim} implies that
    \begin{equation*}
         \phi(e_a)\cdot\xi\cdot\Psi = (-1)^{m+1} e_a \cdot \Psi \, ,
    \end{equation*}
    and using this property we can rewrite
     \begin{equation*}
	   \frac{1}{2}\operatorname{Ric}^g(e_a)\cdot\Psi= \left( (m+1)\delta-\alpha \right)\alpha \, e_a \cdot \Psi   \, ,
    \end{equation*}
    showing the result.
\end{proof}
\begin{rem}
    Note the Ricci curvature is that of an $\eta$-Einstein $\alpha$-Sasaki manifold with $\lambda=2\alpha((m+1)\delta-\alpha)$ and $\nu=2(m+1)\alpha(\delta-\alpha)$.
\end{rem}

At this stage it is illustrative to highlight the similarities between our treatment of $(\alpha,\delta)$-Sasaki manifolds and 3-$(\alpha,\delta)$-Sasaki manifolds. First of all, for an $(\alpha,\delta)$-Sasaki manifold the parameter $\alpha\delta$ determines the transverse Kähler geometry in exactly the same way as for $3$-$(\alpha,\delta)$-Sasaki manifold it determines the transverse quaternionic Kähler geometry.
\begin{rem}
\label{rem:RiccicurvaturebaseSUm}
    From \eqref{eq:riccitensorsu3}, using the O'Neill formulas one finds that the Ricci curvature of the K\"ahler base space is given by
\begin{align*}\label{eq:baseRic}
\operatorname{Ric}^{g_N}(X,Y)&=2\alpha((m+1)\delta-\alpha)g_N(X,Y)+\frac12g(T(X,\xi),T(Y,\xi))\\
&=2\alpha((m+1)\delta-\alpha)g_N(X,Y)+2\alpha^2g(\phi X,\phi Y)=2(m+1)\alpha\delta\, g_{\mathcal{H}}\, ,
\end{align*}
In particular, here the base is Kähler--Einstein. Furthermore, it shows the parameter $\alpha\delta$ is directly related to the curvature of the base space and it motivates the distinction of three different cases in complete analogy with \Cref{def:valuesofalphadelta}:
    \begin{enumerate}[a)]
        \item If $\delta=0$, the manifold is \emph{null} $\alpha$-Sasaki.
        \item If $\alpha\delta>0$, the manifold is \emph{positive} $\alpha$-Sasaki.
        \item If $\alpha\delta<0$, the manifold is \emph{negative} $\alpha$-Sasaki.
    \end{enumerate}
One can therefore regard $\delta=0$ as a ``degenerate'' case where $M$ is a contact Calabi--Yau manifold \cite{Tomassini:2006} and the second equation in \eqref{eq:exteriorderivativesSUm} simplifies to $\dd\Omega=0$. This is analogous to the role the parameter $\delta$ plays for 3-$(\alpha,\delta)$-Sasaki manifolds.
\end{rem}
Contact Calabi--Yau manifolds have recently gathered some attention in relation to string theory models, see \cite{Figueroa-OFarrill:2015gzk, Lotay:2021eog,Aggarwal:2023swe}. One of their main advantages is that in dimension $7$ the associated G$_2$-structure can be studied via 
techniques of Calabi--Yau geometry.

For $(\alpha,\delta)$-Sasaki manifolds we obtain an analogue of \Cref{prop:Hhomotheticdeformations} for an appropriate definition of $\mathcal{H}$-homothetic deformations:
\begin{prop}
\label{prop:HhomotheticforSasaki}
    Let $(M,\xi,\eta,\phi,g,\Phi,\Omega)$ be an $(\alpha,\delta)$-Sasaki manifold. Consider the deformed structure tensors
    \begin{equation*}
        \tilde{g}=c^2g|_\mathcal{V}+ag|_\mathcal{H},\quad \tilde{\eta}=c\,\eta,\quad \tilde{\xi}=\frac1c\xi,\quad \tilde{\phi}=\phi, \quad \tilde{\Omega}=a^{\frac{m}{2}}\Omega\, .
    \end{equation*}
    Then $(M,\tilde{g},\tilde{\xi},\tilde{\eta},\tilde{\phi},\tilde{\Phi},\tilde{\Omega})$ is an $(\tilde{\alpha},\tilde{\delta})$-Sasaki manifold with
    \begin{equation*}
        \tilde{\alpha}=\frac ca\alpha,\qquad\tilde{\delta}=\frac{\delta}{c}\,.
    \end{equation*}
    In particular, the deformed structure is null/positive/negative if and only if the initial structure is.
\end{prop}
\begin{proof}
    Combining $\tilde{g}$ and $\tilde{\phi}$ we find $\tilde{\Phi}=a\,\Phi$, and from \eqref{eq:normalizationSUmstructure} we see that $(\tilde{g},\tilde{\xi},\tilde{\eta},\tilde{\phi},\tilde{\Phi},\tilde{\Omega})$ defines an $\mathrm{SU}(m)$-structure on $M$. Substituting the forms $(\eta,\Phi,\Omega)$ in terms of $(\tilde{\eta},\tilde{\Phi},\tilde{\Omega})$ in \eqref{eq:exteriorderivativesSUm}, we find the stated values of $\tilde{\alpha}$ and $\tilde{\delta}$.
\end{proof}

Since $(\alpha,\delta)$-Sasaki manifolds are $\alpha$-Sasaki, they admit a distinguished connection with skew torsion. This connection can be seen as the analogue of the canonical connection of $3$-$(\alpha,\delta)$-Sasaki manifolds.
\begin{thm}[\cite{Friedrich:2001nh}]
    An $\alpha$-Sasaki manifold $(M,g,\xi,\eta,\phi)$ admits a unique connection $\nabla$ with skew torsion preserving the $\alpha$-Sasaki structure, $\nabla\xi=0$, $\nabla\eta=0$, $\nabla\phi=0$. The torsion tensor is parallel $\nabla T=0$ and is given by
\begin{equation}
\label{eq:torsionSasakiconnection}
T=\eta\wedge\mathrm{d}\eta=2\alpha\,\eta\wedge\Phi\,.
\end{equation}
\end{thm}
\begin{defi}
    We call the connection above the \emph{Sasaki connection} and refer to its torsion, curvature and other associated tensors as Sasaki.
\end{defi}
Note that this connection does not preserve in general the $\mathrm{SU}(m)$-structure since it does not parallelize the associated spinor $\Psi$ \cite{Friedrich:2001nh}. Indeed, using the formula for $\nabla^g\Omega$ in \Cref{lem:lemma1ContiFinno} and the explicit form of the torsion \eqref{eq:torsionSasakiconnection}, we can compute the covariant derivative
\begin{equation*}
    \nabla_\xi\Omega=-i\left( 2m\,\alpha-(m+1)\delta \right)\Omega\, , \qquad \nabla_X\Omega=0 \, , \ \ \text{for all } X\in\mathcal{H} \, .
\end{equation*}
We find that the Sasaki connection preserves the underlying $\mathrm{SU}(m)$-structure if and only if $\delta=\frac{2m}{m+1}\alpha\,$. This subfamily of manifolds where the Sasaki connection parallelizes additional forms is the $(\alpha,\delta)$-Sasaki analogue of parallel $3$-$(\alpha,\delta)$-Sasaki manifolds and it will play an important role later.

\begin{rem}
    We summarize here some specific values of $\alpha$ and $\delta$ that correspond to distinguished types of $(\alpha,\delta)$-Sasaki manifolds.
    \begin{itemize}
        \item The case $\alpha=\delta=1$ corresponds to Sasaki--Einstein manifolds \cite{Sparks:2010sn}. Similarly, the case $\alpha=\delta$ is just a global rescaling of a Sasaki--Einstein manifold.
        \item The degenerate case $\delta=0$ corresponds to (spin) null $\eta$-Einstein $\alpha$-Sasaki manifolds. This is the case of contact Calabi--Yau manifolds first introduced in \cite{Tomassini:2006}.
        \item The case $\delta=\frac{2m}{m+1}\alpha$ corresponds to manifolds where the Sasaki connection preserves the underlying $\mathrm{SU}(m)$-structure \cite{Friedrich:2001nh}. The case with $m=3$ and $\alpha=1$ has been further considered in \cite{Friedrich:2007}.
    \end{itemize}
    These cases are represented schematically in \Cref{fig:alphadeltaSasaki}.
\end{rem}
\begin{figure}[h]
\centering
\begin{tikzpicture}
\draw [->] (-2,0) to (4,0);
\draw [->] (0,-2) to (0,4);
\draw  (-1.5,-1.5) to (3.5,3.5);
\draw  (-1,-1.5) to (2.67,4);
\draw[dotted]  (0,1) to (1,1);
\draw[dotted]  (1,0) to (1,1);
\filldraw[black] (1,1) circle (2pt);
\node at (-0.5,4) {$\delta$};
\node at (4,-0.5) {$\alpha$};
\node at (-0.5,1) {$1$};
\node at (1,-0.5) {$1$};
\node at (2.6,1) {Sasaki--Einstein};
\node [rotate=60] at (1.6,3.3) {$\delta=\frac{2m}{m+1}\alpha$};
\node [rotate=45] at (3.4,3) {$\alpha=\delta$};
\node [gray] at (4,2) {positive};
\node [gray] at (5,0) {null};
\node [gray] at (4,-1.5) {negative};
\end{tikzpicture}
\caption{\label{fig:alphadeltaSasaki}Distinguished values of $(\alpha,\delta)$ for an $(\alpha,\delta)$-Sasaki manifold.}
\end{figure}
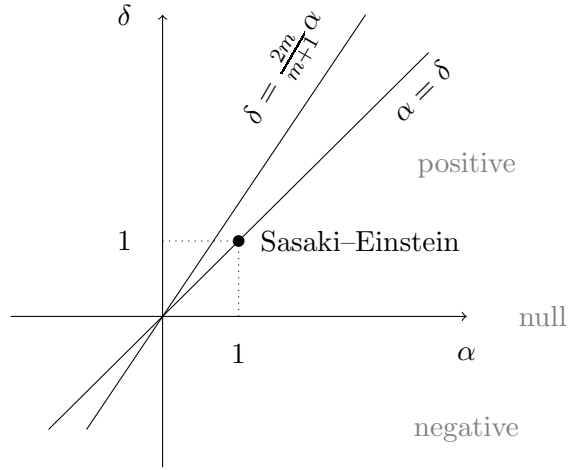

\subsection{The family of connections $\nabla^\lambda$}

In analogy to \Cref{lambdaconn} we look at the following family of connections, which are well-defined not just for $(\alpha,\delta)$-Sasaki manifolds but for any $\alpha$-Sasaki manifold in arbitrary dimension $2m+1$.
\begin{lem}
\label{lem:familyofconnectionsSasaki}
Let $(M,\xi,\eta,\phi,g)$ be an $\alpha$-Sasaki manifold. Then there exists a family of metric connections $\nabla^\lambda$ with $\nabla^\lambda \phi=\nabla^\lambda\xi=\nabla^\lambda\eta=0$ and parallel torsion given by
\begin{align*}
T^\lambda(X,Y,Z)=\begin{cases}T(X,Y,Z) &  X\in \mathcal{V},\ Y,Z\in\mathcal{H}\, ,\\
T(X,Y,Z)-\frac\lambda2(\eta\wedge\Phi)(X,Y,Z) & Y\in \mathcal{V},\ X,Z\in\mathcal{H}\, ,\\
T(X,Y,Z)-\frac\lambda2(\eta\wedge\Phi)(X,Y,Z) & Z\in \mathcal{V},\ X,Y\in\mathcal{H}\, ,\\
0&\text{else}\, .\end{cases}
\end{align*}
These connections are projectable under the locally defined submersion $\pi\colon (M,g)\to(N,g_N)$ introduced in \Cref{thm:sasakiprojection}, satisfying
\begin{equation}\label{projSas}
\nabla^{g_N}_XY=\pi_*(\nabla^\lambda_{\overline{X}}\overline{Y})\, .
\end{equation}
\end{lem}

\begin{proof}
We find that the metric connection with torsion $T^\lambda$ is given by $\nabla^\lambda=\nabla+\Delta^\lambda$, where 
\begin{align*}
\Delta^\lambda(X,Y,Z)&=
\begin{cases}
-\frac\lambda2 (\eta\wedge\Phi)(X,Y,Z) & Y\in \mathcal{V},X,Z\in\mathcal{H}\, ,\\
0 & \text{else}\, .
\end{cases}
\end{align*}
In particular $\nabla^\lambda_Y=\nabla_Y$ for horizontal vectors $Y\in \mathcal{H}$ and \eqref{projSas} follows from the same statement for $\nabla$, see \cite{Stecker:2021}. We also immediately see that $\nabla^\lambda_Y\phi=\nabla_Y\phi=0$ for all cases except $Y\in\mathcal{V}$, $X,Z\in\mathcal{H}$, and in that case we have
\begin{align*}
g(X,(\nabla^\lambda_Y\phi)Z)&=g(X,(\nabla_Y\phi)Z)+\Delta^\lambda(X,Y,\phi Z)+\Delta^\lambda(\phi X,Y,Z)\\
&=\frac\lambda 2\eta(Y)(\Phi(X,\phi Z)+\Phi(\phi X,Z))=0\, .
\end{align*}
Now $\phi(\nabla^\lambda\xi)=-(\nabla^\lambda\phi)\xi=0$ and $0=\mathrm{d}|\xi|^2=2\eta(\nabla^\lambda\xi)$ prove that $\nabla^\lambda\xi=\nabla^\lambda\eta=0$. We then conclude that $T^\lambda$ is parallel.
\end{proof}

\begin{rem}
    We have an $(\alpha,\delta)$-Sasaki version of \Cref{rem:connectionasadeformation}: if $\lambda<4\alpha$ then $\nabla^\lambda$ can be understood as the Sasaki connection of the $(\alpha,\delta)$-Sasaki structure obtained by $\mathcal{H}$-homothetic deformation---in the sense of \Cref{prop:HhomotheticforSasaki}---with the parameter $a$ arbitrary and
\begin{equation*}
c=\sqrt{1-\frac\lambda{4\alpha}}\, .
\end{equation*}
\end{rem}

\begin{rem}
    These connections with the particular choice $\lambda=4\alpha-\frac{m+1}{2m}\delta$ were studied in \cite{Harland:2011zs} for Sasaki--Einstein manifolds and deformations of the form $(\alpha,\delta)=(1,\delta)$ with $\delta>0$. In \cite{Harland:2011zs}, the authors find these connections to be $\mathrm{SU}(m)$-instantons, and we will see below that for $m=3$ these are the only G$_2$-instantons in our family of connections.
\end{rem}

\begin{lem}
Let $Z$ be a horizontal lift of a vector field on $N$, $Y\in\mathcal{V}$ and $X\in\mathcal{H}$. We have
\begin{equation}\label{nablavertSas}
g(X,\nabla^\lambda_YZ)=\left(2\alpha-\frac\lambda2\right)\eta(Y)\Phi(Z,X)\, .
\end{equation}
\end{lem}

\begin{proof}
From \cite[Lemma 2.4]{Stecker:2022} we have that $g(X,\nabla_YZ)=T(X,Y,Z)=2\alpha\,\eta(Y)\Phi(Z,X)$ and thus
\begin{equation*}
g(X,\nabla^\lambda_YZ)=g(X,\nabla_YZ)+\Delta^\lambda(X,Y,Z)=\left(2\alpha-\frac\lambda2\right)\eta(Y)\Phi(Z,X)\, .\qedhere
\end{equation*}
\end{proof}

\begin{prop}
The curvature $R^\lambda(X,Y,Z,V)$ vanishes whenever any vector is vertical and for horizontal vectors
\begin{equation}\label{R-Rgn}
R^{g_N}(X,Y,Z,V)=R^\lambda(\overline{X},\overline{Y},\overline{Z},\overline{V})-\alpha(4\alpha-\lambda)\Phi(\overline{X},\overline{Y})\Phi(\overline{Z},\overline{V})\, .
\end{equation}
In particular, the curvature operator $R^\lambda$ satisfies pair symmetry.
\end{prop}

\begin{proof}
If either $Z\in\mathcal{V}$ or $V\in\mathcal{V}$ then by skew-symmetry in the last entries of $R^\lambda$ and the invariance of $\mathcal{V}$ we find $R^\lambda(X,Y,Z,V)=0$. Using skew-symmetry in the first entries, the same holds for $X,Y\in \mathcal{V}$. Suppose now that $X\in\mathcal{V}$ and $Y,Z,V\in\mathcal{H}$. Then using the Bianchi identity for $\nabla$ we have
\begin{align*}
2R(X,Y,Z,V)&=2R(X,Y,Z,V)-2R(Z,V,X,Y)=\cyclic{X,Y,Z,V}(\cyclic{X,Y,Z}R(X,Y,Z,V))\\
&=\cyclic{X,Y,Z,V}(\sigma_T(X,Y,Z,V)+\cyclic{X,Y,Z}\Delta^\lambda(V,T(X,Y),Z))\\
&=\cyclic{X,Y,Z,V}(\cyclic{X,Y,Z}\Delta^\lambda(V,T(X,Y),Z))=0\, ,
\end{align*}
where the last step follows since $\Delta^\lambda (V,T(X,Y),Z)$ vanishes whenever exactly one of $X,Y,Z,V$ is vertical and the others are horizontal.

We now prove \eqref{R-Rgn}. Let $X,Y,Z,V\in\mathcal{H}$, we have $\nabla^{g_N}_X\nabla^{g_N}_YZ=\pi_*(\nabla^\lambda_{\overline{X}}\nabla^\lambda_{\overline{Y}}\overline{Z})$ and 
\begin{align*}
g_N(V,\nabla^{g_N}_{[X,Y]}Z)&=g(\overline{V},\nabla^\lambda_{\overline{[X,Y]}}\overline{Z})=g(\overline{V},\nabla^\lambda_{[\overline{X},\overline{Y}]}\overline{Z}-\nabla^\lambda_{[\overline{X},\overline{Y}]_\mathcal{V}}\overline{Z})\\
&=g(\overline{V},\nabla^\lambda_{[\overline{X},\overline{Y}]}\overline{Z})-\left(2\alpha-\frac\lambda2\right)\eta([\overline{X},\overline{Y}])\Phi(\overline{Z},\overline{V})\\
&=g(\overline{V},\nabla^\lambda_{[\overline{X},\overline{Y}]}\overline{Z})+\left(2\alpha-\frac\lambda2\right)\mathrm{d}\eta(\overline{X},\overline{Y})\Phi(\overline{Z},\overline{V})\\
&=g(\overline{V},\nabla^\lambda_{[\overline{X},\overline{Y}]}\overline{Z})+\alpha(4\alpha-\lambda)\Phi(\overline{X},\overline{Y})\Phi(\overline{Z},\overline{V})\, ,
\end{align*}
and the statement follows.
\end{proof}

\begin{prop}\label{Rlambdadecomp1}
Let $(M,\xi,\eta,\phi,g,\Phi,\Omega)$ be an $(\alpha,\delta)$-Sasaki manifold, then
\begin{equation}
\label{eq:curvatureSasakifamily}
\mathcal{R}^\lambda=\alpha\left(\left(4\alpha-\frac{m+1}{2m}\delta\right)-\lambda\right)\Phi\otimes\phi+\mathcal{R}_2\, ,
\end{equation}
where $\Phi\in \ker \mathcal{R}_2$ and $\mathcal{R}_2$ is purely horizontal related to $\mathcal{R}^{g_N}$ via
\begin{equation*}
\mathcal{R}^{g_N}=\mathcal{R}_2-\frac{2(m+1)\alpha\delta}{m}\Phi\otimes\phi\, .
\end{equation*}
\end{prop}

\begin{proof}
By \cite[Definition 11.2]{Besse:2008} for Kähler--Einstein manifolds $\Phi$ is an eigenvector of the Kähler curvature operator, with eigenvalue the negative of the Einstein constant, see \Cref{rem:RiccicurvaturebaseSUm}.\footnote{Note the curvature in \cite{Besse:2008} is defined with an additional $-1$ compared to ours.} Since $|\Phi|^2=m$,
\begin{equation*}
\mathcal{R}^{g_N}=-\,\frac{2(m+1)\alpha\delta}{m}\Phi\otimes\phi+\mathcal{R}_2\, ,
\end{equation*}
for some $\mathcal{R}_2$ with $\Phi\in\ker\mathcal{R}_2$. The statement follows from \eqref{R-Rgn}.
\end{proof}

\subsection{Solving the heterotic G$_2$ system}

From this point onwards we focus on the case $m=3$ corresponding to 7 dimensions. The $\mathrm{SU}(3)$-structure gives rise to a $\mathrm{U}(1)$-family of G$_2$-structures with associative and coassociative forms given as follows
\begin{align}
\begin{split}
\label{eq:G2structureSasaki}
    \varphi(\theta) & =-\eta\wedge\Phi +\sin(\theta)\, \Omega_+ +\cos(\theta)\, \Omega_- \, , \\
    \psi(\theta) & =\frac{1}{2}\Phi\wedge\Phi +\sin(\theta)\,\eta\wedge \Omega_- -\cos(\theta)\, \eta\wedge \Omega_+ \, ,
\end{split}
\end{align}
where $\theta\in\mathrm{U}(1)$ and we are using the standard volume form $e^{1\cdots7}=-\frac{1}{3!}\eta\wedge\Phi^3$. See \Cref{app:SUmstructuresandspinors} for a description of these G$_2$-structures in terms of the $\mathrm{SU}(3)$ spinors.

\begin{prop}
    The G$_2$-structures $\varphi(\theta)$ in \eqref{eq:G2structureSasaki} are all coclosed with torsion classes
    \begin{equation*}
    \tau_1=\tau_2=0\, , \quad \tau_0=-\,\frac{4}{7}(3\alpha+4\delta)\, , \quad \tau_3=\frac{4}{7}(\alpha-\delta)\left( 4\eta\wedge\Phi-3\left( \sin(\theta)\, \Omega_+ +\cos(\theta)\, \Omega_- \right) \right) \, .
\end{equation*}
\end{prop}
\begin{proof}
    Using \eqref{eq:exteriorderivativesSUm} it is immediate to compute $\dd\varphi$ and $\dd\psi$ and identify the torsion classes. Alternatively, one can obtain the torsion classes directly from the spinors using the formulas from \cite{Agricola:2014}.
\end{proof}
Using \eqref{eq:torsiong2} we obtain
\begin{cor}
    The torsion of the characteristic G$_2$-connection is
\begin{equation*}
\label{chartorsion}
    T^c(\theta)=\frac{-6\alpha+8\delta}{3}\eta\wedge\Phi+\frac{6\alpha-4\delta}{3}\left( \sin(\theta)\, \Omega_+ +\cos(\theta)\, \Omega_- \right) \, ,
\end{equation*}
and its exterior derivative is
\begin{equation}\label{eq:SU3dtorsion}
    \dd T^c(\theta)=2\alpha\,\frac{-6\alpha+8\delta}{3}\Phi\wedge\Phi+4\delta\,\frac{6\alpha-4\delta}{3}\left(-\sin(\theta)\,\eta\wedge \Omega_- +\cos(\theta)\, \eta\wedge \Omega_+\right) \, .
\end{equation}
\end{cor}

\begin{rem}
    Note that the Sasaki connection and the characteristic connection of an $(\alpha,\delta)$-Sasaki manifold do not generally agree. By uniqueness, they agree if they have holonomy in $\mathrm{SU}(3)= \mathrm{G}_2\cap \mathrm{U}(3)\subset\mathrm{SO}(7)$ which corresponds to the parallel case $\delta=\frac 32 \alpha$. As we already pointed out, this setting was found and investigated in \cite{Friedrich:2007}.

    This constitutes a key difference with the $3$-$(\alpha,\delta)$-Sasaki case, for which the canonical connection is also the characteristic connection.
\end{rem}

We can now investigate which connections in our family are G$_2$-instantons

\begin{thm}
\label{thm:G2instantonsSasaki}
Let $(M,\xi,\eta,\phi,g,\Phi,\Omega)$ be an $(\alpha,\delta)$-Sasaki manifold. The connection $\nabla^{\lambda}$ satisfies 
\begin{equation}
\label{eq:obstructiontoSasakiG2instanton}
    \mathcal{R}^\lambda\wedge\psi(\theta)= \, \frac{\alpha}{2}\left(\frac 43 (3\alpha-2\delta)-\lambda\right)\Phi^3\otimes\phi \, ,
\end{equation}
which means that $\nabla^{\lambda}$ is a G$_2$-instanton for the G$_2$-structure $\varphi(\theta)$ if and only if 
\begin{equation*}
\lambda=\frac 43 (3\alpha-2\delta)\, .
\end{equation*}
In particular, $\nabla$ is a G$_2$-instanton if and only if $\delta=\frac32\alpha$.
\end{thm}

\begin{proof}
We decompose the curvature $\mathcal{R}^\lambda$ as in \Cref{Rlambdadecomp1} and begin showing that $\mathcal{R}_2\wedge\psi(\theta)=0$. Recall from \eqref{eq:G2structureSasaki} that $\psi(\theta)$ has terms depending on $\Omega_+$, $\Omega_-$ and $\Phi\wedge\Phi$. With respect to the complex structure $\phi\vert_{\mathcal{H}}$ we have $\Omega_-\in\Lambda^{0,3}$, $\Omega_+\in\Lambda^{3,0}$ and $\Phi\in \Lambda^{1,1}$. Furthermore, since $\mathcal{R}^{g_N}\colon \Lambda^{1,1}\to\Lambda^{1,1}$ for any Kähler manifold, so does $\mathcal{R}_2$ and we immediately conclude $\mathcal{R}_2\wedge\Omega_+=\mathcal{R}_2\wedge\Omega_-=0$. In addition, $\mathcal{R}_2\perp \Phi$ so $\Phi\wedge\Phi\wedge\mathcal{R}_2=0$, proving the statement.

We are only left with the term proportional to $\Phi\otimes\phi$ in \eqref{eq:curvatureSasakifamily}. Again $\Phi\in \Lambda^{1,1}$ implies that $\Phi\wedge\Omega_-=0$ and $\Phi\wedge\Omega_+=0$, so a single term in $\mathcal{R}^\lambda\wedge\psi(\theta)$ survives. Note it does not vanish: since $\eta$ is a contact form we have
\begin{equation*}
8\alpha^3\eta\wedge\Phi\wedge\Phi\wedge\Phi=\eta\wedge\mathrm{d}\eta^3\neq0 \, ,
\end{equation*}
and a direct computation finishes the proof.
\end{proof}

We are interested in solving the heterotic Bianchi identity using connections from the family $\nabla^\lambda$. To do so, we will see that $\mathrm{d}T^c$ has to be proportional to $\Phi\wedge\Phi$. By \eqref{eq:SU3dtorsion} this implies 
\begin{equation}\label{PhiPhicond}
\delta(3\alpha-2\delta)=0\, .
\end{equation}
Observe that this is precisely satisfied if either $\mathrm{Hol}(\nabla)=\mathrm{SU}(3)$ or $M$ is a degenerate $(\alpha,\delta)$-Sasaki. We obtain the following theorem:

\begin{thm}
\label{thm:solvingBianchiidentitySasaki}
Let $(M,\xi,\eta,\phi,g,\Phi,\Omega)$ be a 7-dimensional $(\alpha,\delta)$-Sasaki manifold with G$_2$-structure $\varphi(\theta)$ as in \eqref{eq:G2structureSasaki}. Consider connections $(V,A)=(TM,\nabla^{\lambda_1})$ and $(TM,\Theta)=(TM,\nabla^{\lambda_2})$ from the family in \Cref{lem:familyofconnectionsSasaki}. Then $[(M,\varphi(\theta)),(V,A),(TM,\Theta),H=T^c]$ solves the heterotic Bianchi identity \eqref{eq:heteroticBianchiidentity} if and only if
\begin{enumerate}[a)]
    \item $\delta=0$ and $\lambda_2=4\alpha\pm\sqrt{(4\alpha-\lambda_1)^2-\frac8{3\alpha'}}$,
    \item $\delta=\frac 32\alpha$ and $\lambda_2=\pm\sqrt{(\lambda_1)^2+\frac8{3\alpha'}}$.
\end{enumerate}
\end{thm}

\begin{proof}
Since $\Phi\perp \mathcal{R}_2$ we find from \eqref{eq:curvatureSasakifamily} that
\begin{align*}
\mathrm{tr}(\mathcal{R}^\lambda\wedge \mathcal{R}^\lambda)&=\frac{\alpha^2}{9}(4(3\alpha-2\delta)-3\lambda)^2\mathrm{tr}(\phi^2)\Phi\wedge\Phi+\tr(\mathcal{R}_2\wedge\mathcal{R}_2)\\
&=-\frac {2\alpha^2}3(4(3\alpha-2\delta)-3\lambda)^2\Phi\wedge\Phi+\tr(\mathcal{R}_2\wedge\mathcal{R}_2)\,.
\end{align*}
In the heterotic Bianchi identity, the curvature terms $\tr(\mathcal{R}_2\wedge\mathcal{R}_2)$ cancel out and the right hand side only has a $\Phi\wedge\Phi$ term. From \eqref{eq:SU3dtorsion} we then see that solutions are only possible if \eqref{PhiPhicond} is satisfied. In this case, we have
\begin{equation*}
\mathrm{d}T^c=\pm4\alpha^2\Phi\wedge\Phi\,,
\end{equation*}
where the $+$ sign corresponds to $\delta=\frac{3}{2}\alpha$ and the $-$ sign corresponds to $\delta=0$. The heterotic Bianchi identity reduces to an algebraic equation for the coefficient of the $\Phi\wedge\Phi$ term
\begin{align*}
    -4\alpha^2 & = \frac{\alpha'}{4}\left(-6\alpha^2 (4\alpha-\lambda_1)^2 + 6\alpha^2 (4\alpha-\lambda_2)^2 \right)   & \text{for } \delta & =0 \, , \\
    4\alpha^2 & = \frac{\alpha'}{4}\left(-6\alpha^2 (\lambda_1)^2 + 6\alpha^2 (\lambda_2)^2 \right)   & \text{for } \delta & =\frac{3}{2}\alpha \, ,
\end{align*}
and solving the second order equation for $\lambda_2$ yields the result.
\end{proof}

Our procedure does not yield exact solutions to the heterotic G$_2$ system:

\begin{thm}
    Let $[(M,\varphi(\theta)),(V,A),(TM,\Theta),H]$ be as in \Cref{thm:solvingBianchiidentitySasaki}. The quadruple is never an exact solution to the heterotic G$_2$ system.
\end{thm}
\begin{proof}
    For an exact solution we need the connections $A=\nabla^{\lambda_1}$ and $\Theta=\nabla^{\lambda_2}$ to be G$_2$-instantons. By \Cref{thm:G2instantonsSasaki} this forces $\lambda_1=\lambda_2=\frac 43 (3\alpha-2\delta)$ and from \Cref{thm:solvingBianchiidentitySasaki} we see the heterotic Bianchi identity can not be solved.
\end{proof}

Nevertheless, we do find approximate solutions as we now show.

\begin{thm}
\label{thm:approximateSasakisolutions}
    Let $[(M,\varphi(\theta)),(V,A),(TM,\Theta),H]$ be as in \Cref{thm:solvingBianchiidentitySasaki}. The quadruple is an approximate solution to the heterotic G$_2$ system if and only if $\delta=\frac{3}{2}\alpha$, $\lambda_1=0$, $\lambda_2=\pm\sqrt{\frac8{3\alpha'}}$ and $\alpha=\mathcal{O}(\alpha')^{\frac{5}{2}}$ when $\alpha'\to 0\,$.
\end{thm}
\begin{proof}
    By \Cref{thm:G2instantonsSasaki}, the connection $A=\nabla^{\lambda_1}$ is a G$_2$-instanton if and only if $\lambda_1=4\alpha$ for $\delta=0$ or $\lambda_1=0$ for $\delta=\frac{3}{2}\alpha$. From \Cref{thm:solvingBianchiidentitySasaki} we immediately see that the heterotic Bianchi identity can not be solved in the case $\delta=0$, whereas for $\delta=\frac{3}{2}\alpha$ we must take
    \begin{equation*}
        \lambda_2=\pm\sqrt{\frac8{3\alpha'}} \,  . 
    \end{equation*}
    In this case it only remains to impose the instanton condition on the connection $\Theta$ up to first order in $\alpha'$. From \eqref{eq:obstructiontoSasakiG2instanton} we have
    \begin{equation*}
        \vert\mathcal{R}^{\lambda_2}\wedge\psi(\theta)\vert_g= \frac{\alpha}{2}\sqrt{\frac8{3\alpha'}}\vert\Phi^3\otimes\Phi\vert_g= 6\alpha\sqrt{\frac6{\alpha'}} \, ,
    \end{equation*}
    and we see the condition \eqref{eq:approximateG2instantoncondition} is satisfied if $\alpha=\mathcal{O}(\alpha')^{\frac{5}{2}}$ when $\alpha'\to 0\,$.
\end{proof}

\begin{rems}
    There are several possible choices of $\alpha$ that provide valid approximate solutions as $\alpha'\to 0$, the simplest one being $\alpha=\left(\alpha'\right)^{\frac{5}{2}}$.

    In all these approximate solutions the geometry of the compact space is tied to the string constant $\alpha'$. In particular, note that as $\alpha'\to 0$ we have $\alpha,\delta\to 0$ which corresponds to a decompactification limit. Similarly, since $\Lambda\sim -\,\tau_0^2=-\left(\frac{35}{7}\alpha\right)^2$, the 3-dimensional spacetime is anti-de Sitter with cosmological constant $\Lambda\to 0$ as $\alpha'\to 0$. Therefore, it approaches flat Minkowski space as $\alpha'\to 0$, analogously to the solutions of \cite{Lotay:2021eog}.
    
    Finally, it was pointed out in \cite{delaOssa:2021cgd} that when the holonomy of the characteristic connection is reduced to $\mathrm{SU}(3)$, solutions of the heterotic G$_2$ system present enhanced $\mathcal{N}=2$ supersymmetry. This is precisely the case for $\delta=\frac{3}{2}\alpha$, so \Cref{thm:approximateSasakisolutions} gives examples of $\mathcal{N}=2$ vacuum solutions where the compact manifold is not a simple direct product of a $6$-dimensional $\mathrm{SU}(3)$-manifold and a circle.
\end{rems}

\section{Conclusion and outlook}
\label{sec:conclusion}

In this paper we have studied the heterotic G$_2$ system on $3$-$(\alpha,\delta)$-Sasaki and $(\alpha,\delta)$-Sasaki manifolds. These are two different types of almost-contact manifolds with reduced structure group, and they admit coclosed G$_2$-structures. In both cases, we have constructed a $1$-parameter family of connections on the tangent bundle that we have later used to solve the heterotic Bianchi identity.

For $3$-$(\alpha,\delta)$-Sasaki manifolds, the associated family of connections includes two different G$_2$-instantons. We have shown that these provide an exact solution of the heterotic G$_2$ system in the degenerate case $\delta=0$. On the other hand, for $\delta\neq 0$ we have only been able to obtain approximate solutions to the system.

We have also obtained approximate solutions using $(\alpha,\delta)$-Sasaki manifolds. Although these manifolds can be understood simply as spin $\eta$-Einstein $\alpha$-Sasaki manifolds, our presentation highlights a similar behaviour to that of $3$-$(\alpha,\delta)$-Sasaki manifolds. This provides a very useful guiding principle for the construction of approximate solutions.

\medskip

There are several potential directions for future research. First of all, it would be interesting to obtain even more solutions to the heterotic G$_2$ system. To this end, one could study families of connections in bundles other than the tangent bundle. Alternatively, one could try to exploit more general G$_2$-structures, particularly in the 3-$(\alpha,\delta)$-Sasaki setting.

One of the key properties of the connections we use in our solutions is their projectability. Therefore, another natural way to look for new solutions of the system would be to consider deformations of the connections that preserve this property. This could potentially describe a particular direction within the moduli space of solutions.

In fact, the infinitesimal moduli space of the heterotic G$_2$ system has received wide attention in the literature recently \cite{delaOssa:2016ivz, Clarke:2016qtg, delaOssa:2017pqy, delaOssa:2017gjq, Fiset:2017auc, Clarke:2020erl}. One could study the infinitesimal moduli space of the solutions we present here and analyze whether they are obstructed or they correspond to honest deformations.

Furthermore, since the spacetime associated to the solutions we present is Anti-de Sitter, a better understanding of the moduli space would be relevant to the Swampland program \cite{Palti:2019pca,vanBeest:2021lhn}. In particular, this would provide a new scenario to test the AdS distance conjecture \cite{Lust:2019zwm} in the heterotic setting.

Although $\eta$-Einstein $\alpha$-Sasaki manifolds have been extensively studied in the literature, our presentation as $(\alpha,\delta)$-Sasaki manifolds is new. Further exploration of this perspective could prove fruitful in a similar way to the case of $3$-$(\alpha,\delta)$-Sasaki manifolds.

Finally, the manifolds we have considered admit a description in terms of spinors. A more detailed analysis of these spinors and their properties would provide further insight into the geometry of $3$-$(\alpha,\delta)$-Sasaki and $(\alpha,\delta)$-Sasaki manifolds.

\section*{Acknowledgements}

We would like to thank Vicente Cort\'es, Ilka Agricola and Xenia de la Ossa for their continued support and interest in the project and Jordan Hofmann for many discussions on spinors.

MG is funded by the Deutsche Forschungsgemeinschaft (DFG, German Research Foundation) under Germany’s Excellence Strategy–EXC 2121 ``Quantum Universe"–390833306.

\appendix

\section{Spinor conventions}
\label{app:spinors}

The appendix expands on the spinor picture of the $G$-structures used in this work, namely $G=\mathrm{G}_2,\ \mathrm{SU}(m)$, and $\mathrm{Sp}(n)$ in dimensions $7$, $2m+1$, and $4n+3$, respectively. This picture is evident throughout the main text, yet does not have the spotlight we feel it deserves.
We begin introducing our notation and conventions for spinors, that mostly follow \cite{Friedrich:1991}. We focus on the odd-dimensional case which is the one relevant to us.

Let $\lbrace e_1,\dots,e_{2m+1}\rbrace$ be the canonical orthonormal basis of $\R^{2m+1}$. The Clifford algebra $\mathrm{Cliff}(\R^{2m+1})$ is the multiplicative algebra generated by the vectors $\lbrace e_1,\dots,e_{2m+1}\rbrace$ together with the relations
\begin{equation}
\label{eq:cliffordrelations}
    e_\mu\cdot e_\nu+e_\nu\cdot e_\mu = -2\delta_{\mu\nu} \qquad \text{for} \ \ \mu,\nu\in\lbrace 1,\dots,2m+1\rbrace \, .
\end{equation}
The group $\mathrm{Spin}(2m+1)$ sits inside $\mathrm{Cliff}(\R^{2m+1})$ as a subgroup
\begin{equation*}
    \mathrm{Spin}(2m+1)=\lbrace x_1\cdot\,\dots\,\cdot x_{2n} \, \vert \, x_\mu\in\mathbb{R}^{2m+1},\, \norm{x_\mu}=1,\, n\in\N \rbrace \, .
\end{equation*}
We can use the complexification of the Clifford algebra $\mathrm{Cliff}^{\mathbb{C}}(\R^{2m+1})\simeq \mathrm{Mat}(2^m,\C)\oplus\mathrm{Mat}(2^m,\C)$ to construct a complex $2^m$-dimensional irreducible representation of $\mathrm{Spin}(2m+1)$. To do so, we define the matrices
\begin{equation}
\label{eq:matricesforspinorreps}
    g_1=\begin{pmatrix}
        i & 0 \\
        0 & -i
    \end{pmatrix}\,, \qquad g_2=\begin{pmatrix}
        0 & i \\
        i & 0
    \end{pmatrix}\,, \qquad E=\begin{pmatrix}
        1 & 0 \\
        0 & 1
    \end{pmatrix}\,, \qquad T=\begin{pmatrix}
        0 & -i \\
        i & 0
    \end{pmatrix}\,,
\end{equation}
and we use them to construct an algebra morphism $\rho:\mathrm{Cliff}^{\mathbb{C}}(\R^{2m+1})\longrightarrow\mathrm{Mat}(2^m,\C)$ as follows:
\begin{align*}
    \rho(e_{1})&=i\, T \otimes \overset{(m)}{\cdots} \otimes T \, , \\
    \rho(e_{2a})&=E\otimes\overset{(m-a)}{\cdots}\otimes E \otimes g_1 \otimes T \otimes \overset{(a-1)}{\cdots} \otimes T \, , \\
    \rho(e_{2a+1})&=E\otimes\overset{(m-a)}{\cdots}\otimes E \otimes g_2 \otimes T \otimes \overset{(a-1)}{\cdots} \otimes T \, ,
\end{align*}
where $a\in\lbrace1,\dots,m\rbrace$ and $\otimes$ denotes the Kronecker product of matrices. Note that our choice of morphism is a reordering of the one described in \cite{Friedrich:1991}. Now, the restriction of $\rho$ to the subgroup $\mathrm{Spin}(2m+1)$ provides an irreducible representation of $\mathrm{Spin}(2m+1)$ on $\C^{2^m}$ that we call the \emph{spinor representation} $\Delta_{2m+1}\,$. The elements of $\C^{2^m}$ are called \emph{spinors}.

The map $\rho$ can be also used to define an action of vectors on spinors known as \emph{Clifford multiplication}. Given a vector $v=v^i e_i$ the Clifford multiplication of $v$ with a spinor $u$ is given by 
\begin{equation*}
    v\cdot u\coloneqq v^i \rho(e_i) u \, .
\end{equation*}
When defining the spinor representation, there is a sign ambiguity in the choice of $\rho(e_{1})$. Our choice is such that for any spinor $u$ we have
\begin{equation}
\label{eq:cliffordproductvolumeform}
    e_1\cdot \, \dots \, \cdot e_{2m+1} \cdot u = (-i)^{m+1} u \, .
\end{equation}
There exists a basis of spinors that is particularly well adapted to our purposes. To construct it, define the following vectors in $\C^2$
\begin{equation*}
    u(1)=\frac{1}{\sqrt{2}}\begin{pmatrix}
        1 \\ -i
    \end{pmatrix}\, , \qquad u(-1)=\frac{1}{\sqrt{2}}\begin{pmatrix}
        1 \\ i
    \end{pmatrix}\, ,
\end{equation*}
and use them to define the spinors
\begin{equation*}
    u(\varepsilon_1,\dots,\varepsilon_m)=u(\varepsilon_1)\otimes \cdots \otimes u(\varepsilon_m)\, ,
\end{equation*}
where $\varepsilon_a=\pm 1$ for $a\in\lbrace 1,\dots,m\rbrace$. Then, a basis of the space of spinors is given by
\begin{equation}
\label{eq:basisofspinors}
    \lbrace u(\varepsilon_1,\dots,\varepsilon_m) \, \vert \, \varepsilon_a=\pm 1\, , \, a\in\lbrace 1,\dots,m\rbrace \rbrace \, .
\end{equation}
This basis is orthonormal with respect to the standard hermitian product $\langle\cdot,\cdot\rangle$ in $\C^{2^m}$, and the hermitian conjugate of $u(\varepsilon_1,\dots,\varepsilon_m)$ is given by $u(-\varepsilon_1,\dots,-\varepsilon_m)^T$. The Clifford multiplication of the vector $e_{1}$ on this basis is particularly simple
\begin{equation*}
    e_{1}\cdot u(\varepsilon_1,\dots,\varepsilon_m) = i \left(\prod_{a=1}^m\varepsilon_a\right)(-1)^m \, u(\varepsilon_1,\dots,\varepsilon_m) \, .
\end{equation*}
Now, let $(M,g)$ be a $2m+1$-dimensional oriented Riemannian spin manifold with spin structure $Q$. We denote the associated spinor bundle by
\begin{equation*}
    \Sigma\coloneqq Q\,\times_{\mathrm{Spin}(2m+1)}\,\Delta_{2m+1}\, .
\end{equation*}
A section of $\Sigma$ is called a \emph{Dirac spinor}. We will typically work with an orthonormal frame $\lbrace e_1,\dots,e_{2m+1}\rbrace$ of $M$ and describe Dirac spinors in terms of the basis $\lbrace u(\varepsilon_1,\dots,\varepsilon_m)\rbrace$.

The manifolds we want to study possess spinors that are particularly well-behaved with respect to covariant derivatives.
\begin{defi}
\label{def:generalizedKillingspinor}
A nowhere-vanishing spinor $\Psi$ is a \emph{generalized Killing spinor} if it satisfies
\begin{equation}
\label{eq:generalizedKillingspinor}
	\nabla^g_X\Psi=\frac{1}{2} S(X)\cdot\Psi \, ,
\end{equation}
where $S$ is a symmetric endomorphism of $TM$.
\end{defi}
Finally, we point out that Clifford multiplication can be used to define a $\C$-linear map $j_\Psi:T^\C M\longrightarrow\Sigma$ \cite{Lawson:1989} (see also \cite{Kopczynski:1997}) as follows
\begin{equation}
\label{eq:mapjpsi}
    j_\Psi(v)\coloneqq v\cdot\Psi \, ,
\end{equation}
for any $v\in T^\C M$. Furthermore, we can define an inner product on $T^\C M$ by linearly extending the metric $g$. It turns out the subspace $\ker(j_\Psi)$ is \emph{isotropic} with respect to this inner product, meaning that $g(v,w)=0$ for any $v,w\in \ker(j_\Psi)$.
\begin{defi}
\label{def:purespinor}
    We say that a spinor $\Psi\in\Gamma(\Sigma)$ is \emph{pure} if the subspace $\ker(j_\Psi)$ is \emph{maximally} isotropic, for the $2m+1$-dimensional case this means $\dim^\C(\ker(j_\Psi))=m$.
\end{defi}

\subsection{$\mathrm{G}_2$-structures in terms of spinors}
\label{app:G2structuresandspinors}

We now focus on the case $m=3$ corresponding to a 7-dimensional manifold $M$. In addition to the spinor bundle obtained via the spinor representation $\Delta_{7}$ on $\C^8$, we can also study spinors on a bundle obtained via a real representation, see for example \cite{Friedrich:2007, Agricola:2014} although our conventions are slightly different. To do so, consider the algebra morphism $\rho^\R:\mathrm{Cliff}(\R^{7})\longrightarrow\mathrm{Mat}(8,\R)$ defined as follows
\begin{align*}
    \rho^\R(e_1) & =E_{12} + E_{34} + E_{56} + E_{78} \, ,  & \rho^\R(e_2) & =E_{13} - E_{24} - E_{57} + E_{68} \, , \\
    \rho^\R(e_3) & =E_{14} + E_{23} + E_{58} - E_{67} \, ,  & \rho^\R(e_4) & =E_{15} - E_{26} + E_{37} - E_{48} \, ,  \\
    \rho^\R(e_5) & =E_{16} + E_{25} - E_{38} - E_{47} \, ,  & \rho^\R(e_6) & =E_{17} - E_{28} - E_{35} + E_{46} \, ,  \\
    \rho^\R(e_7) & =E_{18} + E_{27} + E_{36} + E_{45} \, ,  &  & 
\end{align*}
where $E_{ij}$ denotes the matrix with the entry $(ij)$ equal to $-1$, $(ji)$ equal to $+1$ and all other entries equal to $0$. The restriction of $\rho^\R$ to the subgroup $\mathrm{Spin}(7)$ provides an irreducible representation of $\mathrm{Spin}(7)$ on $\R^8$ that we call the \emph{real spinor representation} $\Delta^{\R}_{7}\,$. The standard basis of $\R^8$ constitutes an orthonormal basis of real spinors. Given $(M,g)$ a 7-dimensional oriented Riemannian spin manifold, we can construct a \emph{real spinor bundle} $\Sigma^\R$ via this real representation
\begin{equation*}
    \Sigma^\R \coloneqq Q\,\times_{\mathrm{Spin}(7)}\,\Delta^\R_{7}\, .
\end{equation*}
A section of $\Sigma^\R$ is called a \emph{Majorana spinor}. These spinors are important to us because they are associated to the existence of G$_2$-structures on $M$.
\begin{prop}
    Let $M$ be a 7-dimensional (oriented) Riemannian spin manifold. A G$_2$-structure on $M$ is equivalent to the choice of a nowhere-vanishing Majorana spinor on the real spinor bundle of $M$ up to scalar multiplication.
\end{prop}
Indeed, given a nowhere-vanishing Majorana spinor $\Psi\in\Gamma(\Sigma^\R)$ the subgroup of $\mathrm{Spin}(7)$ fixing the spinor is precisely G$_2$, and the associative three-form as well as the coassociative four-form can be recovered by
\begin{align}
\label{eq:associativeformintermsofspinors}
    \varphi & =-\sum_{\mu,\nu,\rho}\frac{1}{3!}\langle\Psi,e_\mu\cdot e_\nu\cdot e_\rho\cdot\Psi\rangle\,e^\mu\wedge e^\nu\wedge e^\rho \, , \\
\label{eq:coassociativeformintermsofspinors}
    \psi & =-\sum_{\mu,\nu,\rho,\sigma}\frac{1}{4!}\langle\Psi,e_\mu\cdot e_\nu\cdot e_\rho\cdot e_\sigma\cdot\Psi\rangle\,e^\mu\wedge e^\nu\wedge e^\rho\wedge e^\sigma  \, ,
\end{align}
where $\langle \cdot,\cdot\rangle$ denotes the scalar product in the real spinor bundle induced by the standard $\R^8$ product. On the other hand, if $M$ has a G$_2$-structure the real spinor representation splits into irreducible G$_2$ representations as $\Delta^\R_8=\textbf{1}\oplus\textbf{7}$, where $\textbf{1}$ is the trivial representation and $\textbf{7}$ is the vector representation. We obtain a nowhere-vanishing Majorana spinor patching together the spinors in the $\textbf{1}$ representation.

It is also possible to understand $\Delta^{\R}_{7}$ as a sub-representation of $\Delta_{7}$ and regard Majorana spinors as Dirac spinors satisfying an additional reality condition \cite{Figueroa-OFarrill:notes}. To do so, we define the \emph{charge conjugation matrix} $C$ as
\begin{equation*}
    C=T\otimes E\otimes T \, ,
\end{equation*}
where $T$ and $E$ are defined in \eqref{eq:matricesforspinorreps}. Note $C$ is a real symmetric matrix that squares to the identity and satisfies $C \cdot \rho(e_\mu)=-\rho(e_\mu)^T \cdot C$ for all $\mu\in\lbrace 1,\dots,7\rbrace$. We define the map $J:\C^8\longrightarrow\C^8$ as follows
\begin{equation*}
    J(u) \coloneqq  C \bar{u} \, .
\end{equation*}
The map $J$ is invariant under the action of $\mathrm{Spin}(7)$. In addition, it defines a \emph{real structure} on $\C^8$ in the sense that it is antilinear $J(\lambda\Psi)=\bar{\lambda}J(\Psi)$ and squares to the identity. The spinors satisfying the reality condition $J(u)=u$ give rise to a real subrepresentation of $\Delta_{7}$.

This is most easily seen in a basis. A basis for the real spinors is given by
\begin{align*}
    v_1 &= \frac{1}{\sqrt{2}}(u(1,1,1)+u(-1,-1,-1)) \, ,  &  v_2 &= -i\frac{1}{\sqrt{2}}(u(1,1,1)-u(-1,-1,-1)) \, , \\
    v_3 &= \frac{1}{\sqrt{2}}(u(-1,1,1)-u(1,-1,-1)) \, ,  &  v_4 &= i\frac{1}{\sqrt{2}}(u(-1,1,1)+u(1,-1,-1)) \, , \\
    v_5 &= -\frac{1}{\sqrt{2}}(u(1,-1,1)+u(-1,1,-1)) \, ,  &  v_6 &= -i\frac{1}{\sqrt{2}}(u(1,-1,1)-u(-1,1,-1)) \, , \\
    v_7 &= \frac{1}{\sqrt{2}}(u(1,1,-1)-u(-1,-1,1)) \, ,  &  v_8 &= i\frac{1}{\sqrt{2}}(u(1,1,-1)+u(-1,-1,1)) \, ,
\end{align*}
where we are using the basis of $\C^8$ introduced in \eqref{eq:basisofspinors}. We then find that the representation $\Delta_{7}$ reduces precisely to the representation $\Delta^{\R}_{7}$ on $\R^8=\mathrm{Span}_\R\left( v_1,\dots,v_8 \right)$ that we introduced earlier.

\subsection{\texorpdfstring{$\mathrm{SU}(m)$}{SU(m)}-structures in terms of spinors}
\label{app:SUmstructuresandspinors}

We now discuss in detail how an $\mathrm{SU}(m)$-structure on an odd-dimensional manifold $M$ can be equivalently formulated in terms of the existence of nowhere-vanishing spinors.

Let $M$ be a $2m+1$-dimensional manifold with an $\mathrm{SU}(m)$-structure $(\xi,\eta,\phi,g,\Phi,\Omega)$ (see \Cref{def:SUmstructure}). Then $M$ must be a spin manifold and the inclusion $\mathrm{SU}(m)\subset\mathrm{SO}(2m+1)$ lifts to $\mathrm{SU}(m)\subset\mathrm{Spin}(2m+1)$. The group $\mathrm{SU}(m)$ acts on the spinor bundle $\Sigma$ fixing a 2-dimensional space of Dirac spinors \cite{Harland:2011zs}. To see this explicitly, consider an adapted frame $\lbrace e_1,\dots,e_{2m+1}\rbrace$ of $M$ as defined in \eqref{eq:explicitSU3forms}.
Note the frame satisfies $e_{1}=\xi$ and $e_{2a+1}=\phi(e_{2a})$ for $a\in\lbrace 1,\dots,m\rbrace$. Since $M$ is in particular an almost contact metric manifold, the spinor bundle $\Sigma$ splits as follows:
\begin{lem}[{\cite[Lemma 6.2]{Kim:1999}}]
\label{lem:lemma6.2Kim}
    Let $(M,\xi,\eta,\phi,g)$ be a $2m+1$-dimensional almost contact metric manifold which is spin and with fundamental form $\Phi$. Then the spinor bundle $\Sigma$ splits into the orthogonal direct sum $\Sigma=\Sigma_0\oplus\Sigma_1\oplus\cdots\oplus\Sigma_m\, ,$ where
    \begin{equation*}
        \Phi\vert_{\Sigma_r}=-i\,(2r-m) \operatorname{Id} \, , \qquad  \xi\vert_{\Sigma_r}=i\,(-1)^r(-1)^m \operatorname{Id} \, , \qquad \dim(\Sigma_r)=\begin{pmatrix}
            m \\ r
        \end{pmatrix} \, .
    \end{equation*}
    Moreover, the bundles $\Sigma_0$ and $\Sigma_m$ can be defined by
    \begin{align*}
        \Sigma_0 & = \lbrace \Psi\in\Gamma(\Sigma) \, \vert \, -\phi(X)\cdot\Psi + i\, X\cdot\Psi+(-1)^m\eta(X)\Psi=0 \ \ \text{for all vectors } X\rbrace \, , \\
        \Sigma_m & = \lbrace \Psi\in\Gamma(\Sigma) \, \vert \, -\phi(X)\cdot\Psi - i\, X\cdot\Psi-\eta(X)\Psi=0 \ \ \text{for all vectors } X\rbrace \, .
    \end{align*}
\end{lem}

A basis of $\Sigma_r$ is given by the spinors $u(\varepsilon_1,\dots,\varepsilon_m)$ (as introduced in \eqref{eq:basisofspinors}) for which exactly $r$ elements of $\lbrace \varepsilon_1,\dots,\varepsilon_m \rbrace$ are equal to $-1$. Note as well that with our conventions some additional minus signs appear in \Cref{lem:lemma6.2Kim} when compared with the lemma as stated in \cite{Kim:1999}.

Considering the subgroup $\mathrm{SU}(m)\subset\mathrm{SO}(2m+1)$ that leaves the forms $\eta,\Phi,\Omega$ invariant, it can be checked that this group precisely fixes the spinors in the 2-dimensional bundle $\Sigma_0\oplus\Sigma_m\,$. This implies that both $\Sigma_0$ and $\Sigma_m$ have nowhere-vanishing sections, as the local descriptions can be patched together consistently. We can now compute how the holomorphic volume form acts on these sections
\begin{lem}
\label{lem:actionofholvolformSUm}
    Let $M$ be a $2m+1$-dimensional manifold equipped with an $\mathrm{SU}(m)$-structure $(\xi,\eta,\phi,g,\Phi,\Omega)$. Given the splitting of the spinor bundle $\Sigma=\Sigma_0\oplus\Sigma_1\oplus\cdots\oplus\Sigma_m$ according to the almost contact metric structure and writing $\Omega=\Omega_+ + i\,\Omega_-$, we have
    \begin{equation*}
        \Omega_+\vert_{\Sigma_r}=0 \, , \qquad \Omega_-\vert_{\Sigma_r}=0 \, , \qquad \text{for} \ \ r\in\lbrace 1,\dots,m-1\rbrace \, .
    \end{equation*}
    Let $\Psi\in\Gamma(\Sigma_0)$ nowhere-vanishing, we have that $\bar{\Psi}\in\Gamma(\Sigma_m)$ is also nowhere-vanishing and the following formulas hold:
    \begin{align*}
       \Omega_+\cdot\Psi&=i^m\,2^{m-1}\, \bar{\Psi} \, , & \Omega_-\cdot \Psi&=-i^{m+1}\, 2^{m-1}\, \bar{\Psi} \, ,  \\
       \Omega_+\cdot\bar{\Psi}&=(-1)^{\frac{m(m-1)}{2}}\, i^m \,2^{m-1}\, \Psi \, , & \Omega_-\cdot\bar{\Psi}&=(-1)^{\frac{m(m-1)}{2}}\, i^{m+1} \, 2^{m-1}\, \Psi \, .        
    \end{align*}
\end{lem}
\begin{proof}
    The lemma can be shown by direct computation using the explicit form of $\Omega$ in \eqref{eq:explicitSU3forms} together with Clifford multiplication. The fact that $\bar{\Psi}\in\Gamma(\Sigma_m)$ is immediate looking at the spinors in a local frame.
\end{proof}
The forms associated to the $\mathrm{SU}(m)$-structure can be recovered from $\Sigma_0\oplus\Sigma_m$ as bilinears in the spinors. Given a nowhere-vanishing Dirac spinor $\Psi\in\Gamma(\Sigma_0)$, using the Clifford product formulas in \Cref{lem:lemma6.2Kim} we find that the almost contact structure is recovered by
\begin{equation}
\label{eq:almostcontactformsfromspinors}
    \eta=i\,(-1)^{m+1}\sum_\mu\langle\Psi,e_\mu\cdot\Psi\rangle\,e^\mu\, , \qquad \Phi=-i\,\sum_{\mu,\nu}\frac{1}{2!}\langle\Psi,e_\mu\cdot e_\nu\cdot\Psi\rangle\,e^\mu\wedge e^\nu \, ,
\end{equation}
whereas the formulas in \Cref{lem:actionofholvolformSUm} show that the holomorphic volume form is given by
\begin{equation}
\label{eq:holvolformfromspinors}
    \Omega=(-1)^{\frac{m(m+1)}{2}}\, i^m \sum_{\mu_1,\dots,\mu_m}\frac{1}{m!}\langle\Psi,e_{\mu_1}\cdot \, \dots \, \cdot e_{\mu_m}\cdot\bar{\Psi}\rangle\, e^{\mu_1}\wedge\cdots\wedge e^{\mu_m} \, ,
\end{equation}
where $\langle\cdot,\cdot\rangle$ is the hermitian product on the spinor bundle induced by the hermitian product on the fibres.

Note that nowhere-vanishing sections $\Psi\in\Gamma(\Sigma_0)$ and $\bar{\Psi}\in\Gamma(\Sigma_m)$ are pure in the sense of \Cref{def:purespinor}. Indeed, using the characterization given in \Cref{lem:lemma6.2Kim} we see that $\lbrace e_{2a}+i\, e_{2a+1}\rbrace$ with $a\in\lbrace1,\dots,m\rbrace$ is a basis for $\ker(j_\Psi)$, whereas $\lbrace e_{2a}-i\, e_{2a+1}\rbrace$ with $a\in\lbrace1,\dots,m\rbrace$ is a basis for $\ker(j_{\bar{\Psi}})$. This is the key property of the spinors in $\Sigma_0$ and $\Sigma_m$ that guarantees that they define an $\mathrm{SU}(m)$-structure, as we now illustrate.

\bigskip

Suppose $(M,g)$ is a $2m+1$-dimensional oriented Riemannian spin manifold with a no\-where-vanishing pure Dirac spinor $\Psi\in\Gamma(\Sigma)$. Note that $g$ being a positive metric implies that every element $v\in\ker(j_\Psi)$ must be of the form $v_1+i\, v_2$ for some non-zero $v_1,v_2\in TM$ with $g(v_1,v_1)=g(v_2,v_2)$. This means we can construct an orthonormal frame $\lbrace e_1,\dots,e_{2m+1}\rbrace$ of $M$ such that $\lbrace e_{2a}+i\, e_{2a+1}\rbrace$ with $a\in\lbrace1,\dots,m\rbrace$ is a basis for $\ker(j_\Psi)$. We can then define a $(1,1)$-tensor field $\phi$ by
\begin{equation*}
    \phi(e_{1})=0\,, \qquad \phi(e_{2a})=e_{2a+1} \, , \qquad \phi(e_{2a+1})=-e_{2a} \, , \qquad \text{for} \ \ a\in\lbrace1,\dots,m\rbrace \, .
\end{equation*}
Define $\xi=e_{1}$ and $\eta=e^{1}$. Note that $\phi^2=-\operatorname{Id}+\eta\otimes\xi$, so $(\xi,\eta,\phi)$ defines an almost contact structure on $M$. Furthermore, using that $e_{2a}\cdot\Psi=-i\,e_{2a+1}\cdot\Psi$, the Clifford relations \eqref{eq:cliffordrelations} and the equation \eqref{eq:cliffordproductvolumeform} we deduce that
\begin{equation*}
    \xi\cdot\Psi=(-1)^m\,i\,\Psi\,,
\end{equation*}
which by the characterization of \Cref{lem:lemma6.2Kim} implies that $\Psi\in\Gamma(\Sigma_0)$, where we are decomposing $\Sigma$ with respect to the almost contact structure $(\xi,\eta,\phi)$. Now, we obtain an additional nowhere-vanishing spinor $\bar{\Psi}\in\Gamma(\Sigma_m)$ and we construct the differential forms of the $\mathrm{SU}(m)$-structure via \eqref{eq:almostcontactformsfromspinors} and \eqref{eq:holvolformfromspinors}. Therefore, a nowhere-vanishing pure Dirac spinor determines an $\mathrm{SU}(m)$-structure on $M$.

The discussion of this appendix can be summarized in the following proposition
\begin{prop}\label{prop:SUmspinor}
    Let $(M,g)$ be a $2m+1$-dimensional (oriented) Riemannian spin manifold. An $\mathrm{SU}(m)$-structure on $M$ is equivalent to the existence of a nowhere-vanishing pure Dirac spinor on the spinor bundle of $M$.
\end{prop}
We are specially interested in the case where the spinor associated to the $\mathrm{SU}(m)$-structure is generalized Killing. This should be indicated in the torsion classes, although proofs only exist under additional assumptions. In \cite{Conti:2007} the authors show the following for real analytic manifolds.
\begin{prop}[\cite{Conti:2007}]\label{prop:ContiFino}
    Suppose that $M$ is real analytic. The spinor asso\-ciated to an $\mathrm{SU}(m)$-structure $(\xi,\eta,\phi,g,\Phi,\Omega)$ on $M$ is generalized Killing if and only if
    \begin{equation*}
        \dd \Phi=0\qquad\text{and}\qquad \dd(\eta\wedge\Omega)=0\, .
    \end{equation*}
\end{prop}
The authors in \cite{Conti:2007} argue that real analyticity should not be required for the above statement. However, to our knowledge a proof without real analyticity exists only in dimension $2m+1=5$, see \cite{Conti:2008}. If the structure gives rise to a generalized Killing spinor then the torsion of the $\mathrm{SU}(m)$-structure is encoded by the symmetric endomorphism $S$.

\begin{lem}[{\cite[Lemma 1]{Conti:2007}}]
\label{lem:lemma1ContiFinno}
     Let $(M,g)$ be a $2m+1$-dimensional oriented Riemannian spin manifold with an $\mathrm{SU}(m)$-structure $(\xi,\eta,\phi,g,\Phi,\Omega)$ determined by a generalized Killing spinor $\Psi$ satisfying \eqref{eq:generalizedKillingspinor}. Then, the following holds
    \begin{align*}
        \nabla^g_X \eta & = (-1)^{m+1} \, S(X)\lrcorner\Phi \, , \\
        \nabla^g_X \Phi & = (-1)^{m+1} \, \eta\wedge S(X)^\flat \, , \\
        \nabla^g_X \Omega & = (-1)^{m+1}\, i \,\left( g(S(X),\xi)\,\Omega -\eta\wedge\left( S(X)\lrcorner\Omega \right)\right) \, .
    \end{align*}
\end{lem}
Note that with our choice of conventions these formulas include an additional global minus sign compared to the ones in \cite{Conti:2007}. The proof of \Cref{lem:lemma1ContiFinno} in \cite{Conti:2007} relies on expressing the part of the connection one-form $\omega\in\Lambda^1\otimes\mathfrak{so}(2m+1)$ that takes values in the complement of $\mathfrak{su}(n)$ in terms of the endomorphism $S$. This information determines the covariant derivative of the Reeb vector field and provides some additional information about the covariant derivative of the other elements of an adapted basis:
\begin{lem}
\label{lem:extensionofContiFinnotoframe}
    Let $(M,g)$ be a $2m+1$-dimensional oriented Riemannian spin manifold with an $\mathrm{SU}(m)$-structure $(\xi,\eta,\phi,g,\Phi,\Omega)$ determined by a generalized Killing spinor $\Psi$ satisfying \eqref{eq:generalizedKillingspinor}. The Reeb vector field satisfies
    \begin{equation*}
        \nabla^g_X\xi  = (-1)^m \, \phi(S(X)) \, .
    \end{equation*}
    Consider an adapted frame $\lbrace e_1,\dots,e_{2m+1}\rbrace$ of $M$ in the sense of \eqref{eq:explicitSU3forms}. Then,
    \begin{equation*}
        g(\nabla^g_X e_a, \xi) = (-1)^m  g(S(X),\phi(e_a)) \, , 
    \end{equation*}
    where $a\in\lbrace 2,\dots,2m+1\rbrace$.
\end{lem}

Let us now focus on the case $m=3$ corresponding to a 7-dimensional manifold. Since $\mathrm{SU}(3)\subset \mathrm{G}_2\,$, the $\mathrm{SU}(3)$-structure on $M$ must give rise to a G$_2$-structure. In fact, we have a $\mathrm{U}(1)$ family of G$_2$-structures, as we now illustrate. From the spinor $\Psi\in\Gamma(\Sigma)$ defining the $\mathrm{SU}(3)$-structure, we construct
\begin{equation}
\label{eq:spinorsplusandminus}
    \Psi_+=\frac{1}{\sqrt{2}}(\Psi+\bar{\Psi})\, , \qquad  \Psi_-=-i\frac{1}{\sqrt{2}}(\Psi-\bar{\Psi})\, .
\end{equation}
These are Majorana spinors as introduced in \Cref{app:G2structuresandspinors}. To see this, it is enough to consider the adapted orthonormal frame $\lbrace e_1,\dots,e_{7}\rbrace$ where $\Psi$ is described by $u(1,1,1)$ and observe that in this basis $\Psi_+$ and $\Psi_-$ correspond to the Majorana spinors $v_1$ and $v_2$ in \Cref{app:G2structuresandspinors}. Combining them we obtain a $\mathrm{U}(1)$ family of Majorana spinors that we can parametrize by
\begin{equation*}
    \Psi(\theta)=\cos(\frac{\theta}{2})\Psi_++\sin(\frac{\theta}{2})\Psi_-\,,
\end{equation*}
and using \eqref{eq:associativeformintermsofspinors} and \eqref{eq:coassociativeformintermsofspinors} we find the associative and coassociative forms \eqref{eq:G2structureSasaki} that define the G$_2$-structures induced by these spinors.

\subsection{$\mathrm{Sp}(n)$-structure in terms of spinors}
\label{app:Spnstructuresandspinors}

We conclude by explaining how an $\mathrm{Sp}(n)$-structure on a $4n+3$-dimensional manifold $M$ can be described in terms of spinors. Some references discussing the topic from the perspective of 3-Sasaki geometry are \cite{Friedrich:1990zg, Harland:2011zs, Agricola:2022, Hofmann:2022}.

Let $M$ be a $4n+3$-dimensional manifold with an almost 3-contact metric structure as in \Cref{def:ac3s} and \Cref{rem:acm3s}. This is equivalent by \Cref{rem:acm3sisSpm} to a choice of $\mathrm{Sp}(n)$-structure on $M$. This implies that $M$ is spin \cite{Boyer:2007book}, and in fact the group $\mathrm{Sp}(n)$ fixes a $2n+2$-dimensional space of Dirac spinors on $\Sigma$ \cite{Harland:2011zs}. We will however focus in a particular subset of these spinors.

Recall that by \Cref{rem:acsisUm} an almost contact metric structure in a manifold of dimension $4n+3$ is equivalent to a $\mathrm{U}(2n+1)$-structure. Since $\mathrm{Sp}(n)\subset\mathrm{SU}(2n+1)\subset\mathrm{U}(2n+1)$, the spinors fixed by the $\mathrm{Sp}(n)$-structure include the spinors defining the $\mathrm{SU}(2n+1)$-structures underlying each of the three almost contact metric structures. 

Note that for us the dimension of $M$ is $4n+3=2m+1$ so $m=2n+1$ is odd. Therefore, we can describe each almost contact structure not in terms of $(\Psi,\bar{\Psi})$ but in terms of their real and imaginary parts $(\Psi_+,\Psi_-)$ defined as in \eqref{eq:spinorsplusandminus}. The almost contact forms can be obtained (for odd $m$) by
\begin{equation}
\label{eq:almostcontactformsfromspinorsplusminus}
    \eta=\sum_\mu\langle\Psi_-,e_\mu\cdot\Psi_+\rangle\,e^\mu\, , \qquad \Phi=-\sum_{\mu,\nu}\frac{1}{2!}\langle\Psi_-,e_\mu\cdot e_\nu\cdot\Psi_+\rangle\,e^\mu\wedge e^\nu \, .
\end{equation}
Using \Cref{lem:lemma6.2Kim}, we see that for odd $m$
\begin{equation}
\label{eq:oneformsactingonspinorsplusminus}
    \xi\cdot\Psi_+=\Psi_- \, , \qquad X\cdot\Psi_+=\phi(X)\cdot\Psi_- \ \ \text{for } X\in\langle\xi\rangle^\perp \, .
\end{equation}
In particular, given the spinor $\Psi_+$ we can obtain $\Psi_-$ by a Clifford product with the Reeb vector field $\xi$. Combining both formulas in \eqref{eq:oneformsactingonspinorsplusminus}, we have that
\begin{equation}
\label{eq:horizontalvectoronpsiplus}
    \xi\cdot X\cdot\Psi_+=\phi(X)\cdot\Psi_+ \ \ \text{for } X\in\langle\xi\rangle^\perp \, .
\end{equation}
Therefore, each of the three almost contact metric structures $(\xi_i,\eta_i,\phi_i,g)_i$ is determined by a pair of spinors $(\Psi_{i,+},\Psi_{i,-})$ that recover the associated forms as bilinears via \eqref{eq:almostcontactformsfromspinorsplusminus}.

Thus, the set of six spinors $\lbrace\Psi_{i,\pm}\rbrace_{i=1,2,3}$ fully recovers the almost 3-contact metric structure. Using \eqref{eq:oneformsactingonspinorsplusminus} it is easy to verify that the spinors $\Psi_{i,\pm}$ are precisely the ones spanning the bundles $E_i$ introduced in \cite{Friedrich:1990zg}:
\begin{equation*}
    E_i=\lbrace \Psi\in\Gamma(\Sigma) \, \vert \, \left( -2\,\phi_i(X) + \xi_i\cdot X - X\cdot\xi_i\right) \cdot\Psi=0 \ \ \text{for all vectors } X \rbrace \, .
\end{equation*}
As illustrated in specific examples in \cite{Agricola:2022, Hofmann:2022}, the bundle $E=E_1+E_2+E_3$ might not be a direct sum. This means the spinors $\lbrace\Psi_{i,\pm}\rbrace_{i=1,2,3}\,$ might not all be linearly independent.

This is precisely the situation for the case $n=1$ corresponding to a 7-dimensional manifold, in which we focus from now on. As explained in \cite{Harland:2011zs}, $\mathrm{Sp}(1)$ fixes $4$ spinors on $\Sigma$, and we will show that only $3$ of them belong to $E$. We work in an adapted frame as defined in \eqref{eq:explicitSpnforms}, and describe the spinors in terms of the decomposition of \Cref{lem:lemma6.2Kim}  associated to the first almost contact structure.

The spinors $\Psi_{i,\pm}$ can all be expressed in terms of the spinors $\Psi_{1,\pm}$ and the almost 3-contact forms. In fact, for the 7-dimensional case it is easy to check directly from the Clifford algebra that
\begin{equation*}
    \Psi_{2,+}=\Psi_{1,-}\, , \qquad \Psi_{2,-}=\xi_2\cdot\Psi_{1,-}\, , \qquad \Psi_{3,+}=\xi_2\cdot\Psi_{1,-}\, , \qquad \Psi_{3,-}=\Psi_{1,+}\, .
\end{equation*}
This prompts us to define the \emph{auxiliary spinors}
\begin{equation*}
    \psi_{1}\coloneqq\xi_2\cdot\Psi_{1,-}\, , \qquad \psi_{2}\coloneqq\Psi_{1,+}\, , \qquad \psi_{3}\coloneqq\Psi_{1,-}\, ,
\end{equation*}
so that $E_i$ is spanned by $\psi_j$ and $\psi_k$, and $E$ is spanned by the three spinors $\lbrace\psi_i\rbrace_{i=1,2,3}\,$. There is an additional spinor preserved by the $\mathrm{Sp}(1)$-structure that we call the \emph{canonical spinor}
\begin{equation*}
    \psi_{0}\coloneqq-\xi_2\cdot\Psi_{1,+}\, .
\end{equation*}
Using \eqref{eq:oneformsactingonspinorsplusminus} and \eqref{eq:horizontalvectoronpsiplus} one can show that the following formulas hold:
\begin{align}
\label{eq:formulasSp1canonicalspinor}
    \xi_i\cdot\psi_0&=\psi_i \, & \Phi_i\cdot\psi_0&=\xi_i\cdot\psi_0\, ,  & & \\    
\label{eq:formulasSp1auxiliaryspinors}
    \xi_i\cdot\psi_j&=\psi_k \, , & \Phi_i\cdot\psi_i&=\xi_i\cdot\psi_i \, , & \Phi_i\cdot\psi_j&=-3\,\xi_i\cdot\psi_j \, .
\end{align}
Both the canonical and the auxiliary spinors are Majorana, and thus give rise to G$_2$-structures on $M$ via \eqref{eq:associativeformintermsofspinors}. In fact, by taking combinations of the auxiliary spinors we see that the spinors in the bundle $E$ generate an $\mathrm{SU}(2)$ family of G$_2$-structures. These are fundamentally different from the one generated by the canonical spinor, which is the one we are interested in.

\providecommand{\href}[2]{#2}\begingroup\raggedright\endgroup

\end{document}